\pdfoutput=1
\newif\ifpersonal
\personaltrue 
\RequirePackage[l2tabu,orthodox]{nag} 

\documentclass{amsart}

\usepackage{amsmath,amsthm, amsfonts, amssymb, mathrsfs, bbm, mathtools}
\usepackage[pagewise]{lineno}

\usepackage[hidelinks]{hyperref}
\usepackage[capitalize]{cleveref} 
\usepackage{tikz, tikz-cd} 
\usetikzlibrary{positioning}
\usetikzlibrary{decorations.markings}
\usepackage{verbatim}
\usepackage{microtype,fixltx2e,lmodern} 
\usepackage[utf8]{inputenc} 
\usepackage[T1]{fontenc} 
\usepackage{enumerate,comment,braket,xspace,csquotes} 
\usepackage[all,cmtip]{xy} 
\usepackage{multirow}
\usepackage{stmaryrd}

\usepackage{xcolor}
\hypersetup{
	colorlinks,
	linkcolor={red!50!black},
	citecolor={blue!50!black},
	urlcolor={blue!80!black}
}

\theoremstyle{plain}
\newtheorem{proposition}{Proposition}[section]

\newtheorem{theorem}[proposition]{Theorem}
\newtheorem*{theorem*}{Theorem}
\newtheorem{corollary}[proposition]{Corollary}
\newtheorem*{corollary*}{Corollary}
\newtheorem{conjecture}[proposition]{Conjecture}
\newtheorem{question}[proposition]{Question}

\theoremstyle{definition}
\newtheorem{definition}[proposition]{Definition}
\newtheorem{remark}[proposition]{Remark}

\newtheorem{example}[proposition]{Example}
\newtheorem{notation}[proposition]{Notation}

\ifpersonal
\newcommand*{\personal}[1]{\textcolor[rgb]{.2,.2,1}{\tiny{(Personal: #1)}}}
\newcommand*{\todo}[1]{{\tiny\textcolor{red}{(Todo: #1)}}}
\else
\newcommand*{\personal}[1]{\ignorespaces}
\newcommand*{\todo}[1]{\ignorespaces}
\fi

\usetikzlibrary{decorations.markings}

\makeatletter
\tikzcdset{
	open/.code     = {\tikzcdset{hook, circled};},
	closed/.code   = {\tikzcdset{hook, slashed};},
	open'/.code    = {\tikzcdset{hook', circled};},
	closed'/.code  = {\tikzcdset{hook', slashed};},
	circled/.code  = {\tikzcdset{markwith = {\draw (0,0) circle (.375ex);}};},
	slashed/.code  = {\tikzcdset{markwith = {\draw[-] (-.0ex,-.4ex) -- (.0ex,.4ex);}};},
	markwith/.code ={
		\pgfutil@ifundefined%
		{tikz@library@decorations.markings@loaded}%
		{\pgfutil@packageerror{tikz-cd}{You need to say %
				\string\usetikzlibrary{decorations.markings} to use arrows with markings}{}}{}%
		\pgfkeysalso{/tikz/postaction = {
				/tikz/decorate,
				/tikz/decoration={markings, mark = at position 0.5 with {#1}}}
		}
	},
}
\makeatother

\newcommand{\rH}{\mathrm H}

\newcommand{\fc}{\mathfrak{c}}

\newcommand{\cC}{\mathcal C}
\newcommand{\cD}{\mathcal D}
\newcommand{\cE}{\mathcal E}
\newcommand{\cF}{\mathcal F}
\newcommand{\cH}{\mathcal H}

\newcommand{\cM}{\mathcal M}

\newcommand{\cO}{\mathcal O}
\newcommand{\cP}{\mathcal P}

\newcommand{\cT}{\mathcal T}

\newcommand{\bZ}{\mathbb{Z}}							  				
\newcommand{\bN}{\mathbb{N}}

\newcommand{\bA}{\mathbb{A}}						
\newcommand{\bF}{\mathbb{F}}

\newcommand{\bQ}{\mathbb{Q}}
\newcommand{\bQb}{{\overline{\bQ}}}
\newcommand{\bFb}{{\overline{\bF}}}
\newcommand{\bC}{\mathbb{C}}
\newcommand{\bR}{\mathbb{R}}

\newcommand{\bG}{\mathbb{G}}

\newcommand{\Rep}{\operatorname{Rep}}

\newcommand{\rHB}{\rH_{\mathrm{Betti}}}

\newcommand{\one}{\mathbbm{1}}
\newcommand{\Mot}{{\cM ot}}
\newcommand{\Mat}{{\cM at}}
\newcommand{\Motgr}{{\Mot(\bQb)^\textrm{gr}}}
\newcommand{\MotgrQ}{{\Mot(\bQb)_{\bQ}^\textrm{gr}}}

\newcommand{\Vect}{\mathrm{Vect}}
\newcommand{\obF}{{\overline{\bF}}}

\newcommand{\bQl}{{\bQ_\ell}}

\newcommand{\obQp}{{\overline{\bQ}_p}}
\newcommand{\obQ}{{\overline{\bQ}}}
\newcommand{\cl}{\mathsf{cl}}

\newcommand{\DM}{\mathrm{DM}}

\newcommand{\dR}{\mathrm{dR}} 
\newcommand{\crys}{\mathrm{crys}}
\newcommand{\num}{\mathrm{num}}
\newcommand{\ho}{\mathrm{hom}}

\newcommand{\eval}{\mathrm{eval}}
\newcommand{\fcB}{\mathfrak{c}_{\mathrm{Bert}}}
\newcommand{\cAp}{\mathcal{P}^{\mathrm{mot}}_p}

\newcommand{\Forg}{\mathrm{Forg}}

\newcommand{\alg}{\mathrm{alg}}

\newcommand{\et}{_\mathrm{\acute{e}t}}

\makeatletter
\newcommand{\oset}[3][0.2ex]{%
	\mathrel{\mathop{#3}\limits^{
			\vbox to#1{\kern-2\ex@
				\hbox{$\scriptstyle#2$}\vss}}}}
\makeatother

\newcommand{\<}{\langle}
\renewcommand{\>}{\rangle}

\newcommand{\gr}{\mathsf{gr}}

\newcommand{\ad}{\mathsf{ad}}

\newcommand{\inv}{^{-1}}

\usetikzlibrary{decorations.markings} 
\tikzset{
  closed/.style = {decoration = {markings, mark = at position 0.5 with { \node[transform shape, xscale = .8, yscale=.4] {/}; } }, postaction = {decorate} },
  open/.style = {decoration = {markings, mark = at position 0.5 with { \node[transform shape, scale = .7] {$\circ$}; } }, postaction = {decorate} }
}
\newcommand{\PGL}{\operatorname{PGL}}
\newcommand{\GL}{\operatorname{GL}}
\newcommand{\SL}{\operatorname{SL}}

\newcommand{\Frob}{\operatorname{Frob}}

\providecommand{\fgl}{\mathfrak{gl}}

\newcommand{\fp}{\mathfrak{p}}

\newcommand{\fh}{\mathfrak{h}}

\DeclareMathOperator{\Gal}{Gal}
\DeclareMathOperator{\Aut}{Aut}
\DeclareMathOperator{\Hom}{Hom}
\DeclareMathOperator{\End}{End}
\DeclareMathOperator{\cEnd}{\cE nd}
\DeclareMathOperator{\cHom}{\cH om}
\DeclareMathOperator{\Image}{Im}

\DeclareMathOperator{\Isom}{Isom}
\DeclareMathOperator{\Emb}{Emb} 
\DeclareMathOperator{\Nat}{Nat}

\DeclareMathOperator{\Spec}{Spec}

\DeclareMathOperator{\Sym}{Sym}

\DeclareMathOperator{\id}{\mathsf{id}}

\DeclareMathOperator{\spe}{sp}
\DeclareMathOperator{\MT}{MT}

\newcommand{\Betti}{\mathrm{Betti}}
\newcommand{\CH}{\operatorname{CH}}

\DeclareMathOperator{\trdeg}{trdeg}

\newcounter{dragos}
\newcounter{giuseppe}

\newcommand{\Addresses}{{
		\bigskip
		\footnotesize
		G.~Ancona, \textsc{IRMA, Strasbourg, France}\\ \nopagebreak
		\texttt{ancona@math.unistra.fr}
		
		\medskip
		
		D. Fr\u{a}\c{t}il\u{a}, \textsc{IRMA, Strasbourg, France}\\\nopagebreak
		\texttt{fratila@math.unistra.fr}
}}

\numberwithin{equation}{section}
\title{Algebraic classes in mixed characteristic and André's $p$-adic periods}

\author{Giuseppe Ancona, Drago\c s Fr\u a\c til\u a}

\date{\today}

\begin{document}

	\maketitle
	\begin{abstract}
		Motivated by the study of algebraic classes in mixed characteristic, we define a countable subalgebra of $\bQb_p$ which we call the algebra of Andr\'e's $p$-adic periods. The classical tannakian formalism cannot be used to study these new periods. Instead, inspired by ideas of Drinfel'd on the Plücker embedding and further developed by Haines, we produce an adapted tannakian setting which allows us to bound the transcendence degree of Andr\'e's $p$-adic periods  and to formulate the $p$-adic analog of the Grothendieck period conjecture.
		We exhibit several examples where special values of classical $p$-adic functions appear as Andr\'e's $p$-adic periods and we relate these new conjectures to some classical problems on algebraic classes.
		\end{abstract}
		
			\tableofcontents

\section{Introduction}
\subsection{Motivation} As Tate himself pointed out,  the Tate conjecture predicts the existence of algebraic classes but does not predict which cohomological classes should  be algebraic \cite[Aside 6.5]{MilneTate}. 
More precisely, let $X$ be a smooth projective variety over a field of finite type. Denote by $\CH(X)$ the Chow group of $X$ and by $\rH_\ell(X)$ the $\ell$-adic cohomology of $X$. 
Then we have the cycle class   map
\[\cl_X\colon \CH(X) \longrightarrow \rH_{\ell}(X),\]
and the prediction is that the  
 $\bQ_\ell$-vector  space $\Image (\cl_X\otimes  \bQ_\ell )$  should coincide with the Galois invariant part of the $\ell$-adic cohomology $\rH_\ell(X)$. 
However, it does not describe the $\bQ$-vector space $\Image (\cl_X\otimes  \bQ )\subset \rH_\ell(X)$ and actually the inclusion \[\Image (\cl_X\otimes  \bQ ) \subset \Image (\cl_X\otimes  \bQ_\ell )\] 
is very mysterious.

For example, in the proof of the fact that when $A$ is an abelian variety over a finite field, the natural injection
$\End(A)\otimes \bQ_\ell \hookrightarrow \End_{\Gal}(\rH^1_\ell(A))$
is also surjective \cite{Tate}, the crucial  elements in   $\End(A)\otimes \bQ_\ell$ that   Tate constructed  are actually $\ell$-adic limits of endomorphisms that do not belong  to $\End(A)\otimes \bQ$. In fact, even though Tate's theorem says that abelian varieties over finite fields have plenty of endomorphisms, it is in general very hard to construct explicit ones besides the Frobenius.

Another example of the  subtle difference between    $\Image (\cl_X\otimes  \bQ )$ and $\Image (\cl_X\otimes  \bQ_\ell )$ appears also in the study of the Kuga--Satake construction for the proof of the Tate conjecture for K3 surfaces over finite fields. There, one would like to construct sufficiently many divisors on the K3 surface $S$ by reducing modulo $p$ the divisors on several lifitings  $\tilde{S}$. To do so, the key point is to understand the position, inside the cohomology of $S$, of the Hodge filtration induced by a given lifting with respect to $\Image (\cl_X\otimes  \bQ )$, where $X$ is the Kuga--Satake variety associated with  $S$.

As a very concrete instance of this subtlety, notice that one can find in $\ell$-adic (or crystalline) cohomology plenty of Galois-invariant lines without fundamental classes of subvarieties. For example, consider a  surface $S$ over a finite field which admits two line bundles $L_1$ and $L_2$ on $S$ such that their classes $\cl_S(L_1)$ and $\cl_S(L_2)$ in $\ell$-adic cohomology  are linearly independent and whose intersection product is non zero (for example, $S$ is an abelian surface). Take a non rational scalar $\alpha \in \bQ_\ell $. Then the $\bQ_\ell$-line generated by $\cl_S(L_1)+\alpha \cdot \cl_S(L_2)$ is Galois-invariant but it cannot contain the class of a line bundle as the intersection product between classes of  line bundles is a rational number.

In characteristic zero there is a clear candidate for the $\bQ$-vector subspace $\Image (\cl_X\otimes  \bQ )$.
Namely, the Artin comparison theorem gives an isomorphism between $\ell$-adic cohomology and singular cohomology, and $\Image (\cl_X\otimes  \bQ )$  is described by the Hodge conjecture. Again in \cite[Aside 6.5]{MilneTate}, Tate  suggests that the absence of a conjectural description of $\Image (\cl_X\otimes  \bQ )$ in positive characteristic is one  reason for which algebraic classes are less understood there, even in codimension one (in contrast to the characteristic zero situation where we have the Lefschetz $(1,1)$ theorem). This paper is intended as a first step towards such a description.
The setup that we study is the case where the variety has a lifting to characteristic zero and the cohomology is the crystalline one  (the $\ell$-adic setting is discussed in   \cref{section:ellGPC} but our investigation   is less satisfying).

\subsection{From algebraic classes to periods}
For simplicity in this introduction, let $X_0$ be a smooth projective variety over $\bQ$ and let $p$ be a prime number   of good reduction.
Fix a model $X$ of $X_0$ over some localization of $\Spec(\bZ)$  and let $X_p$ be the fiber of $X$ over the prime $p$ which we assume moreover to be smooth. 
In order to describe the $\bQ$-vector subspace  $\Image (\cl_{X_p}\otimes  \bQ )\subset  \rH^*_\crys(X_p,\bQ_p)$, one can use the 
Berthelot comparison theorem \cite[Theorem V.2.3.2]{BerthelotCrys}    between the de Rham cohomology of $X_0$ and the crystalline cohomology of $X_p$:
\begin{align}\label{BO intro}
	  \rH^*_{\dR}(X_0/\bQ)\otimes\bQ_p  \simeq \rH^*_\crys(X_p,\bQ_p)
\end{align}
and deduce a natural map 
\begin{align}\label{BO point intro}
	\Image (\cl_{X_p}\otimes  \bQ ) \hookrightarrow \rH^*_{\dR}(X_0/\bQ)\otimes\bQ_p .
\end{align}
This map relates two different $\bQ$-vector spaces: algebraic classes in characteristic $p$ with rational coefficients  and de Rham cohomology in characteristic zero. Our original task to describe the position of  $\Image (\cl_{X_p}\otimes  \bQ )$ inside $  \rH^*_\crys(X_p,\bQ_p)$ is equivalent, in this mixed characteristic setting, to describing the position of  $\Image (\cl_{X_p}\otimes  \bQ )$ with respect to $\rH^*_{\dR}(X_0/\bQ)$.
For example, one can fix a $\bQ$-basis of  $\rH^*_{\dR}(X_0/\bQ)$ and try to write the coordinates of the elements of $\Image (\cl_{X_p}\otimes  \bQ )$ with respect to this basis.
Note that these coordinates are numbers in $\bQ_p$ and if we change the fixed basis of $\rH^*_{\dR}(X_0/\bQ)$, the $\bQ$-span of these coordinates in $\bQ_p$ does not change.

Our investigations led us to think that these coordinates behave like periods, both for their transcendental properties and for their relations to special values.
Our evidence is the following: 
\begin{enumerate}
\item[a)] Special values of the $p$-adic logarithm, products of special values of the $p$-adic Gamma function and the $p$\nobreakdash-adic hypergeometric functions appear as such coordinates (see \cref{ex log,Eg:H_M for E=CM,Eg:two ell curves}).
\item[b)] The transcendence degree of these coordinates can be bounded in terms of motivic Galois groups  (\cref{SS:homogeneous})  and it is possible to state  an analog of the Grothendieck period conjecture (\cref{SS:pGPC}), which for elliptic curves appears to be related to a conjecture of Andr\'e (\cref{construction andre}).
\end{enumerate}

These coordinates for all possible $X$ span a countable $\bQ$-subalgebra of $\bQ_p$ which, in light of the above properties, seems to be a  $p$-adic analogue of classical periods.
This algebra is unrelated to Fontaine's ring of $p$-adic periods, which has several spectacular cohomological properties but for which the naive Grothendieck period conjecture \cite{AndreBetti} does not hold. More precisely, it is shown in loc. cit. that its veracity depends on the choice of an embedding $\obQ\hookrightarrow\obQ_p$.
However, more recently André has proved (see \cite[Theorem 3.5]{Andobserv}) that for a generic such embedding, the surjectivity in the Tate conjecture implies the $p$-adic Grothendieck conjecture (for Fontaine's periods). 
As we will explain in \S \ref{related work}, other attempts to construct  such a subalgebra of  $\bQ_p$ of $p$-adic periods appeared in the literature. 
Our approach not only includes and generalizes the previous ones but, more importantly, it is the first one providing a sharp bound of the transcendence degree and which allows a formulation of a $p$-adic analog of  the Grothendieck period conjecture.

\subsection{Classical periods}\label{classical setting}
In the classical setting, periods are complex numbers which can be written as integrals of algebraic differential forms on simplices of algebraic varieties defined over a number field. The cohomological interpretation is the following. Following de Rham and Grothendieck, for a variety $X$ over $\bQ$ we have that integration gives a functorial isomorphism 
\begin{align}\label{Groth point}
	 \rH^*_{\dR}(X/\bQ)\otimes\bC \simeq \rH^*_{\Betti}(X,\bQ)\otimes\bC
\end{align}
and the coordinates of one $\bQ$-structure with respect to the other are precisely the periods $\mathcal{P}_\bC(X)$ associated with $X$. 

Many interesting complex numbers appear as periods, such as $2\pi i$, special values of the logarithm, products of special values of the Gamma function, and one would like to be able to control their transcendence degree. To do so, the idea of Grothendieck was to consider all possible functorial isomorphisms 
\[\rH^*_{\dR}(X/\bQ)\otimes R\simeq \rH^*_{\Betti}(X,\bQ)\otimes R,\] 
for all commutative $\bQ$-algebras $R$, which are moreover compatible with the K\"unneth formula. 
The general tannakian formalism will ensure that there is a variety $\mathcal{T}_X$ over $\bQ$ representing all such comparison isomorphisms. 
Notice that the isomorphism \eqref{Groth point} corresponds to a  $\bC$-point of $\mathcal{T}_X$ with residue field equal to the field generated by $\cP_\bC(X)$, hence the transcendence degree of the algebra generated by  $\mathcal{P}_\bC(X)$ is at most the dimension of $\mathcal{T}_X$. Grothendieck has conjectured that they should be equal, i.e., that such a bound is optimal.
It turns out that this bound, together with all the tannakian interpretation that we sketched above, is at the heart of several recent advances in the arithmetic study of multizeta-values, as for example in \cite{brown}.

In order to make the transcendence bound useful, one needs to compute the dimension of $\mathcal{T}_X$. To do so, Grothendieck has considered all possible functorial automorphisms 
\[\rH^*_{\dR}(X/\bQ)\otimes R \simeq \rH^*_{\dR}(X/\bQ)\otimes R,\]
 for all commutative $\bQ$-algebras $R$, which are compatible with the K\"unneth formula. 
This turns out to be representable by an algebraic group $G(X)$, called  the motivic Galois group of $X$, whose natural action on $\mathcal{T}_X$ by right composition makes $\mathcal{T}_X$ a~$G(X)$-torsor. 
In particular, $\mathcal{T}_X$ and $G(X)$ have the same dimension. Moreover, the group $G(X)$ can be explicitly calculated in numerous cases providing thus explicit bounds for the transcendence degree of the algebra generated by the classical periods~$\cP_\bC(X)$.

\subsection{Main results}
 
In our setting   we would like to study
\begin{align}\label{BO point intro 2}
	\Image (\cl_{X_p}\otimes  \bQ )\otimes\bQ_p  \hookrightarrow \rH^*_{\dR}(X_0/\bQ)\otimes\bQ_p 
\end{align}
induced by the Berthelot comparison isomorphism \eqref{BO point intro}. Our $p$-adic periods  $\mathcal{P}_p(X)$ are defined as the coordinates of the left $\bQ$-structure with respect to the right one. In order to study their transcendence properties one would like to imitate  Grothendieck's construction we sketched above, with the isomorphism   \eqref{Groth point} replaced by the injection \eqref{BO point intro 2}. The fact that this injection   is not in general an isomorphism does not allow one to use the same tannakian constructions and let us think for a longtime that such a tannakian framework adapted to $p$-adic periods could not exist. This is probably also the reason why all the previous approaches to $p$-adic periods in the literature have considered only  settings  where \eqref{BO point intro 2} is an isomorphism (see \S \ref{related work}).

The tannakian interpretation of the Plücker embedding and a generalization of it due to Drinfel'd (see \cite{Hai} for a detailed account),  inspired us to think that such a construction might also apply in our setting, namely considering injections instead of isomorphisms. 
We believe that this psychological step, as natural as it can appear afterwards, is a real contribution of our work.  
More precisely, consider all possible functorial injections 
\[\Image (\cl_{X_p}\otimes  \bQ )\otimes R \hookrightarrow \rH^*_{\dR}(X_0/\bQ)\otimes R\] 
for all commutative $\bQ$-algebras $R$, which are compatible with the tensor structures (see \cref{D:of H_M} for a precise definition) and show that there is a variety $\mathcal{H}_X$ over~$\bQ$ representing all such injections (\cref{T:cH_M is repr}).  
As in the classical case, the map \eqref{BO point intro 2} is a $\bQ_p$-point of $\mathcal{H}_X$ and one deduces the following result:
\begin{theorem}\label{thm intro}
	The transcendence degree of the algebra generated by $p$-adic periods  satisfies the bound
	\begin{align}\label{ineq intro}
	\trdeg_\bQ( \mathcal{P}_p(X) )\le 	   \dim \mathcal{H}_X.
	\end{align}
\end{theorem}

A second difference with the classical case comes from the fact that  the algebra  generated by the $p$-adic periods of $X$ does not contain, in general, all the $p$-adic periods of all powers of $X$. 
This comes from the fact that the K\"unneth formula holds for cohomology groups but does not hold for algebraic classes (as in general powers of $X$ have more algebraic classes than just products of algebraic classes on~$X$). 
For this reason one should also consider the larger algebra $\cP_p(\<X\>)$ generated by the $p$-adic periods of all powers $X^n$ and the tannakian constructions we described above actually give the following inequalities
	\[\trdeg_\bQ ( \mathcal{P}_p(X) )\le 	\trdeg_\bQ (\cP_p(\<X\>))\le    \dim \mathcal{H}_X.\]
	
We also construct a finer tannakian object which gives a sharper bound for the $p$-adic periods of $X$; see
\cref{motivic periods} for details. 
All these bounds   improve the previous ones and give conceptual explanations to several known relations on $p$-adic special values; see \cref{related work}.

As for classical periods \cref{classical setting}, one needs to be able to compute the dimension of $\cH_X$ in order for the bound to be useful.
To do so we consider the motivic Galois group  $G(X_0)$ introduced in \S \ref{classical setting}. It has again a natural action on $\mathcal{H}_X$ by right composition. Now there is the third difference with respect to the classical case:   $\mathcal{H}_X$ is not in the general a $G(X_0)$-torsor but we show that it is always a~$G(X_0)$-homogeneous space. 

In order to compute $\dim \mathcal{H}_X$, one has to describe a stabilizer. We compute the stabilizer precisely of  the $\bQ_p$-point induced by  \eqref{BO point intro 2}. It turns out that such a $\bQ_p$-algebraic subgroup is the motivic Galois group  $G(X_p)$ associated with crystalline cohomology\footnote{Note that  $G(X_p)$ is indeed a subgroup of $G(X_0)\otimes \bQ_p$, through the isomorphism \eqref{BO intro}. }, namely  the $\bQ_p$-algebraic  group representing all functorial, tensor automorphisms $\rH^*_\crys(X_p,\bQ_p)\otimes R \simeq \rH^*_\crys(X_p,\bQ_p)\otimes R$, for all $\bQ_p$-algebras $R$.
To sum up we obtain
	\begin{align}\label{sabato mattina}
\dim \mathcal{H}_X  = \dim G (X_0)  -  \dim G (X_p).
	\end{align}

 The $p$-adic analog of the Grothendieck period conjecture  predicts that the inequality $	\trdeg (\cP_p(\<X\>))\le    \dim \mathcal{H}_X$ should actually be an equality (\cref{Conj:p adic analog GPC tensor}). 
An analogous conjecture is formulated for the subalgebra   $\cP_p(X) \subset \cP_p(\<X\>)$ (see \cref{C:trdegM<=dimSpecA_p} and \cref{Conj:p adic analog GPC}). To be precise, these conjectures are formulated only under a minor condition, called ramification condition, which is briefly discussed in \S \ref{related work}   and detailed in  \cref{SS:ramification}.
As already mentioned, our setting is the first where the Grothendieck period conjecture has a $p$-adic analog; see  \S \ref{related work}.

As particular cases, these conjectures predict the precise transcendence degree of some special values (logarithm, hypergeometric, Gamma function, \ldots) which will hopefully be useful to number theorists. As pointed out in \S \ref{classical setting}, tannakian techniques have been at the heart of several recent advances on rationality questions  on special values (for example multiple zeta values as in \cite{brown}) and we hope that the basis built in this paper might play the same role in the study of  $p$-adic special values.

%

\subsection{Related work}\label{related work}This paper suggests that there should be a bridge  between 
  algebraic classes in  characteristic $p$
and a $p$-adic analog of classical periods.

During our work we discovered that a similar bridge was crossed by André in the opposite direction:  the motivation and the conclusion are reversed \cite{Andre1}.
His aim was to find a canonical $\bQ$-structure of $\rH^*_\crys(X_p,\bQ_p)$ and he realized that, when $X_p$ is a supersingular abelian variety, this is possible using algebraic classes. 
He made a conjecture in this setting which is implied by our analog of the Grothendieck period conjecture (it is a special case of the weak version of our conjecture, as presented in \cref{section:question}).
He realized afterwards that in order for this conjecture to be reasonable one needs to impose some ramification condition above $p$ \cite[I.5.3]{Andre-p}. 
We propose a generalization of it and show that for a given variety $X_0/\bQ$ this condition is satisfied for all but finitely many primes $p$ (\cref{SS:ramification}).
As it appears clear, Andr\'e's work became a great source of inspiration for us and for this reason we would like to call the periods we introduce in this paper Andr\'e's  $p$-adic periods.

In the setting of mixed Tate motives analogous constructions already exist \cite{Furusho,Brownp}.  These works considered the action of the Frobenius on the crystalline realization and defined the $p$-adic periods of the motive as the coefficients of this action with respect to a rational basis induced by \eqref{BO intro}.
It turns out that  Andr\'e's  $p$-adic periods for mixed Tate motives come indeed from those coefficients. However, for general motives, the coefficients of the Frobenius matrix are only special examples of Andr\'e's  $p$-adic periods; see  \cref{more than frob,weight}. Moreover, even just for the coefficients of the Frobenius matrix our definition of $p$-adic periods seems to be the right one for several reasons; see  \cref{remark ultima} for details.
Probably the most important one is the following: the coefficients of the Frobenius matrix  satisfy some algebraic relations due to the fact that the characteristic polynomial of the Frobenius has integral coefficients.
 Those relations have a motivic interpretation in our setting (\cref{pol char}) whereas they do not have one in the setting of  \cite{Furusho,Brownp}; see \cref{remark ultima 2} for more details. In particular, there is no reasonable $p$-adic Grothendieck period conjecture in their setting.

	 
	After a conversation between Joseph Ayoub and the authors, the three of us think that the algebra of motivic $p$-adic periods we construct  in this paper should be related to the algebra of abstract $p$-adic periods constructed in \cite{ayoubp}. In that article Ayoub conjectured that his algebra should be the algebra of functions on a homogeneous motivic space but not on a motivic torsor. It is exactly what happens in our constructions. A precise relation between these two algebras will be studied in the future. This would be the $p$-adic analogue of \cite[Fait 1.4]{ayoubannals}, namely a comparison between the algebra of motivic periods coming from the tannakian setting and the algebra of motivic periods constructed in the triangulated setting.

\subsection{Organization of the paper}
The weak version of the classical Grothendieck period conjecture allows one to formulate concrete conditions predicting the existence of algebraic classes. We start the paper with its $p$-adic analog and discuss the relation between this new conjecture and some classical ones (\cref{section:question}).

For the strong version of the conjecture we need to introduce the motivic framework  (\cref{SS:important hom=num or section of NUM}).
In this setting we can construct the algebra of Andr\'e's $p$-adic periods associated to a motive (\cref{SS:periods}).
 In  \cref{SS:homogeneous} we construct the homogeneous space $\cH_X$ and show that it is representable.

As already mentioned, a ramification condition is needed in order to have a reasonable Grothendieck period conjecture. 
This is discussed in \cref{SS:ramification}.
The ramification condition and the construction of the torsor space allow us to define the $p$-adic analog of the Grothendieck period conjecture  (\cref{SS:pGPC}).

 The paper is written in the language of pure motives but most of the work applies also to the mixed setting. In \cref{S:MTM}, we briefly outline which changes are needed in this context.
 Finally, in  \cref{section:ellGPC} we discuss the $\ell$-adic version of these questions, where \eqref{BO intro} is replaced by Artin's comparison theorem between singular and $\ell$-adic cohomology. 
 In this case, despite some positive results, the general picture is less satisfying.
	
\subsection*{Acknowledgments}	
We thank Emiliano Ambrosi, Yves Andr\'e, Joseph Ayoub, Francis Brown, Ishai Dan Cohen, Pierre Colmez, David Corwin, Clément Dupont, Benjamin Enriquez, Javier Fres\'an, Marco Maculan and Alberto Vezzani for their interest in our work and for numerous fruitful discussions.

We warmly thank the referee for a careful reading which spotted tons of typos and for  all the insightful comments which highly improved the mathematical quality and the exposition of this paper.

 This research was partly supported by the grant ANR--18--CE40--0017 of Agence National de la Recherche. 
\section{Conventions}\label{S:conventions}

Throughout the paper we will work with the following notation and conventions.

\begin{enumerate}

\item \textbf{Base fields.} We fix a prime number $p$ and an algebraic closure $\overline{\bQ}_p$ of ${\bQ}_p$. Let $\overline{\bZ}_p \subset \overline{\bQ}_p$ be the ring of algebraic $p$-adic integers. 
Define  $ \overline{\bQ} \subset \overline{\bQ}_p$ to be the algebraic closure of $\bQ$ in $\overline{\bQ}_p$.
We write $\cO$ for the ring of algebraic integers. 
The inclusion  $ \cO \subset \overline{\bZ}_p$ induces 
 a maximal ideal $\fp$ of  $\cO$ above the ideal  $(p)$ of $\bZ$.
 We denote by $\cO_{\fp}$ the localization of $\cO$ at $\fp$ and by $\overline{\bF}_p$ the residue field.
 
\item\label{crystalline} \textbf{Crystalline cohomology.} Fix a smooth projective variety $X_p$ defined over $\overline{\bF}_p$. 
The usual crystalline cohomology of $X_p$ has coefficients in (the ring of integers of) $\hat{\bQ}^{ur}_p$ the completion of  $ \bQ^{ur}_p$, the maximal unramified extension of $\bQ_p$ in $\bQb_p$. 
We will work instead with crystalline cohomology with $\bQb_p$ coefficients, which is defined as follows. First, consider   the limit of the crystalline cohomology of all models of $X_p$ defined over finite fields.
Such a  limit defines a Weil cohomology with coefficients in $ \bQ^{ur}_p$. Then, one extends  the coefficients from $ \bQ^{ur}_p$ to $\bQb_p$ and obtains a 
  Weil cohomology with coefficients in $\bQb_p$.
  (The construction of such a Weil cohomology is analogous to \cite[\S 2.2, \S 3]{addezio} and is implicit in \cite[I.5.3.5]{Andre-p}.)

Assume moreover that $X_p$ is the reduction modulo $p$ of a smooth projective scheme $X$ over $\cO_{\fp}$ with generic fiber $X_0$. Then  one has a canonical comparison
\begin{align} 
	 \rH^*_{\dR}(X_0/\bQb)\otimes\obQ_p \simeq \rH^*_\crys(X_p,\obQ_p)
\end{align}
which is obtained from   \cite[Theorem V.2.3.2]{BerthelotCrys}  
through the analogous limit process.

\item \textbf{Motives.} We will work with homological motives defined over $\cO_{\fp}$ and their restriction to  $ \overline{\bQ}$ and to $\overline{\bF}_p$. 
More precisely, for a base $S=\cO_{\fp}, \overline{\bQ}$, or $\overline{\bF}_p$ and a field of coefficients of characteristic zero $L$, we let $\Mot(S)_L$ denote the category of homological motives over $S$ with coefficients in $L$ with respect to a Weil cohomology with coefficients containing $L$. For generalities see \cite[IV]{Andmot}.

The category $\Mot(S)_L$ depends a priori on the chosen Weil cohomology. 
In characteristic zero all classical Weil cohomologies (singular, de Rham, $\ell$-adic) give rise to the same category because of the existence of comparison theorems.
In positive characteristic we will always work with one given cohomology, which will be crystalline (except in \cref{section:ellGPC}).

In order to have realization functors with values in usual vector spaces (not super-vector spaces with the Koszul sign commutativity) we  will work only with motives verifying the sign Künneth conjecture (see \cite[Section 6]{Andmot} for details).
We abuse notation and still denote by $\Mot(S)_L$ the category of motives with the commutativity constraint modified.

In this paper we will   work with $L=\bQb$ up to  \cref{SS:homogeneous} and with  
$L=\bQ$ from  \cref{SS:ramification} on.  
When  $L=\bQb$  we drop it from the notation and we simply write \[\Mot(S)=\Mot(S)_{\bQb}.\]

\item\label{rem:cycleclass} \textbf{Cycle class map.}
For a motive $M$, we think of the vector space 
\[\Hom_{\Mot(S)}(\one, M)\] 
as the space of  algebraic classes on $M$. 
Indeed, if $F$ are the coefficients of the cohomology and $R$ is the associated realization functor any map $f\in\Hom_{\Mot(S)}(\one, M)$ is characterized by the image of $1\in R(\one)=F$ in $R(M)$. 
When $M$ is the motive of a variety, this element is an algebraic class in the usual sense. 
The map \[\cl_M \colon \Hom_{\Mot(S)}(\one, M)\longrightarrow R(M) \hspace{1cm} f \mapsto R(f)(1)\] is in this case the usual cycle class map.

\item\label{specialization} \textbf{Specialization functor.} There are two restriction functors
\[ \begin{tikzcd}
	& \Mot(\cO_{\fp})_L \ar[rd,"M\mapsto M_0"] \ar[ld,"M\mapsto M_p",swap]& \\
	\Mot(\bFb_p)_L & & \Mot(\bQb)_L
\end{tikzcd}
 \]
 which are faithful (again by comparison theorems).
 The right hand side functor is moreover full (because any algebraic cycle is the restriction of its closure).
 This allows us to identify $\Mot(\cO_{\fp})_L$ with a full subcategory of $ \Mot(\bQb)_L$ which we denote by
 \[ \Mot(\bQb)^\textrm{gr}_L \subset \Mot(\bQb)_L\]
 and call it the category of motives with good reduction. Again, when  $L=\bQb$  we drop the subscript from the notation.
 
  The left hand functor above induces then a  specialization functor
\[\spe\colon \Mot(\bQb)^\textrm{gr}_L \longrightarrow \Mot(\bFb_p)_L. \]
 	
\item\label{hom=num} \textbf{hom=num.}
Since it is not known that  homological motives form an abelian category, we will work only with motives satisfying  the
	standard conjecture $\sim_\ho = \sim_\num$ both in characteristic $0$ and $p$; see \cref{rem:homnum}. For instance, motives generated by products of elliptic curves do verify all these conjectures \cite{Lieb,Spiess}.

	This hypothesis guarantees that the categories of motives that we consider are abelian and hence tannakian. 
It  implies as well the Lefschetz standard conjecture for those motives  \cite{Smirnov}, hence   it implies the K\"unneth standard conjecture.

For a variety $X$ we will write $\mathfrak{h}(X)$ for its motive and
\[\mathfrak{h}(X)=\bigoplus_i \mathfrak{h}^{i}(X)  \]
for its K\"unneth decomposition. It is intended that the realization of $\mathfrak{h}^{i}(X) $ is concentrated in cohomological degree $i$.

We refer to \cite[Section 6]{Andmot} for details on these conjectures and their implications.

\item\label{not:generation}\textbf{Tannakian subcategories.}
Given an object $X$ in a tannakian category we denote by $\<X\> $ the full tensor subcategory  generated by $X$, closed under direct sums, subquotients and duals. We call it the tannakian category generated by $X$.
If $X_1,\dots,X_n$ are objects in a tannakian category then we put $\<X_1,\dots,X_n\>:=\<X_1\oplus\dots\oplus X_n\>$ and call it the tannakian category generated by the objects $X_1,\dots,X_n$.

\item\label{extending scalars}\textbf{Extending scalars.}
The extension of scalars for tannakian categories $\cC\mapsto \cC\otimes_K L$ is a little technical and we will not give the definition here since it is not essential to the main body of our work. 
The interested reader could consult the following references for details \cite[\S 4]{Delignelog}.

For example, if $G$ is an affine algebraic group over $K$ and $L$ is a field extension of $K$, then we have $\Rep(G)\otimes_K L\simeq\Rep(G_L)$ (\cite[4.6]{Delignelog}).
Under the hypothesis we work on, i.e,  \eqref{hom=num}, the categories of motives we consider are of the above type.

The only properties that we will use are the following. 
If $\cC$ is a tannakian category over $K$ and $L$ is a field extension of $K$, then the scalar extension $\cC\otimes_KL$ comes with a functor
\[ \iota\colon \cC\to \cC\otimes_ KL \text{ verifying}\]
\[ \Hom_\cC(X,Y)\otimes_K L \simeq \Hom_{\cC\otimes_KL}(\iota(X),\iota(Y)).\]
Moreover, for any $X\in \cC$, we have an equivalence $\<X\>\otimes_KL\simeq\<\iota(X)\>$.
\end{enumerate}

\section{Some questions on algebraic classes in mixed characteristic}\label{section:question}

In this section we would like to raise some new questions on algebraic classes in mixed characteristic and relate them to classical conjectures. 

We let $X$ be a smooth projective scheme over an open subset of the ring of integers of some number field.
We denote the generic fiber by $X_0$ (it is a smooth projective variety over a number field $K$) and  by $X_p$ the fiber above a fixed place $\fp$ (it is a smooth projective variety over a finite field).

Berthelot \cite[Theorem V.2.3.2]{BerthelotCrys}  has defined a natural isomorphism between the de Rham cohomology of $X_0$ and the crystalline cohomology of $X_p$:
\begin{align}\label{Eq:the Berthelot compar}
	 \rH^*_{\dR}(X_0/K)\otimes K_{\fp} \simeq \rH^*_\crys(X_p,K_{\fp})
\end{align}
where $K_{\fp}$ is the $\fp$-adic completion of $K$.
The left hand side has a $K$-structure, namely $\rH^*_{\dR}(X_0/K)$.
The right hand side contains a natural $\bQ$-vector space $\rH^*_\alg(X_p)$ the algebraic classes (with rational coefficients) on $X_p$.
These do not, in general, span the whole crystalline cohomology.

\begin{question}(pGPCw)\label{Conj:liftp}
	With the above assumptions, let $\gamma\in \rH^*_\alg(X_p)$ be an algebraic class in the crystalline cohomology of $X_p$ living in the subspace $\rH^*_\dR(X_0/K)\subset \rH^*_\dR(X_0/K)\otimes K_{\fp}$ under Berthelot's isomorphism \eqref{Eq:the Berthelot compar}.
	Does then $\gamma$ lift to an algebraic class in characteristic zero? Namely, is it true that $\gamma$ is in the image of the specialization map $\rH^*_\alg(X_0)\to \rH^*_\alg(X_p)$?
\end{question}

\begin{remark}\label{rem:bcconj}
Let us compare now \cref{Conj:liftp}, which will be denoted  ($p$GPCw), with three classical conjectures which we recall  here in an informal way (see \cite[VII]{Andmot} for more details).
These three conjectures are the Hodge conjecture (HC), the weak form of the Grothendieck period conjecture (GPCw) and the Fontaine--Messing $p$-adic variational Hodge conjecture ($p$HC).    
	
\begin{itemize}
	\item[(HC)] A rational class in singular cohomology is algebraic if and only if it is in the right step of the de Rham filtration.
	\item[(GPCw)] A rational class in de Rham cohomology is algebraic if and only if it is  rational for singular cohomology.
	\item[($p$HC)] An algebraic class in crystalline  cohomology lifts to characteristic zero if and only if it is in the right step of the de Rham filtration.
\end{itemize}
The above three conjectures together with \cref{Conj:liftp} naturally fit into a table
\begin{center}
\begin{tabular}{c|c}
($p${GPCw}) & ({GPCw})\\
\\
\hline
\\
($p${HC}) & ({HC})
\end{tabular}
\end{center}
 The right part of the table is the characteristic zero part, the left part is the mixed characteristic one, the bottom part is the filtration part and the top part is the rationality part. 
 
 The conjecture (GPCw) is implied by the strong form of the Grothendieck period conjecture (concerning transcendence of periods); see \cite[Proposition 7.5.2.2]{Andmot}. (Notice that \textit{loc. cit.} applies under our hypothesis; see \S \ref{S:conventions}\eqref{hom=num}.)
 Its $p$-adic  counterpart is discussed in \cref{SS:pGPC}.
\end{remark}

Using the motivic language 
one can reformulate \cref{Conj:liftp} as follows:
\begin{question}\label{Conj:correspondancep}
Does an algebraic correspondence defined over a finite field lift to characteristic zero as soon as its action on de Rham cohomology is rational?
\end{question}
The following is a special case.
\begin{question}\label{Conj:abelvarp}
Does an endomorphism of an abelian variety over a finite field lift (up to isogeny) to characteristic zero as soon as its action on de Rham cohomology is rational?
\end{question}
Here is a little evidence.
\begin{proposition}
Let $E$ be a non CM elliptic curve defined over  a quadratic number field $K$. Consider a place of good reduction over an inert prime $\fp$ and suppose that the reduction is a supersingular elliptic curve $E_p$ over $\mathbb{F}_{p^2}$. 
Then there is at least one endomorphism of $E_p$  whose action on the de Rham cohomology is not rational.
\end{proposition}
\begin{proof}
Let us argue by contradiction and suppose that the algebra $\End{E_p}$ acts on $\rH^1_{\dR}(E/K)$. We will construct a non trivial endomorphism of $E$ and thus contradict the non CM property.

First notice that since $E_p$ is supersingular, the algebra $\End{E_p}\otimes K$ has dimension four over $K$ and this forces it to be equal to $\End(H^1_\dR(E/K))\simeq M_2(K)$.
In this matrix algebra the elements that respect the de Rham filtration form a $K$-subalgebra $A$ of codimension one.
Let $\alpha$ be a $K$-linear form on $\End E_p\otimes K$ such that $\ker(\alpha)=A$. 

Second, consider the $\bQ$-algebra $B:=A \cap \End{E_p}\otimes \bQ$ and let us show that it is of dimension at least two. 
Indeed, as $K$ is quadratic, the restriction of $\alpha$ to $\End{E_p}\otimes \bQ$   is equivalent to the datum of two $\bQ$-linear forms on $\End{E_p}\otimes \bQ$. Their common kernel is precisely $B$ and therefore the dimension of $B$ is at least $4-2=2$.

In conclusion there is an endomorphism $f\in \End{E_p}\otimes \bQ$ which is not a multiple of the identity and whose action respects the de Rham filtration. 
By \cite[Theorem 3.15]{BO-crys} such an $f$ lifts to characteristic zero. 
This gives a contradiction as we supposed that $E$ was not CM.
\end{proof}
\begin{remark}
\cref{Conj:abelvarp} already appeared in some form and is due to Andr\'e  \cite[Remark in \S 5]{Andre1}.
 As he observed later  \cite[I.4.6.4 and Conjecture I.5.3.10]{Andre-p}, a ramification condition on $p$ is necessary to avoid some counterexamples (we briefly outline his argument in \cref{Eg:Andre CM ell curve}).
	In particular, a similar condition is needed for \cref{Conj:liftp} to be reasonable; we discuss it in \cref{SS:ramification}. 
	Let us just mention here that a variety over $\bQb$ verifies this condition for all but finitely many primes.
	The corrected version of \cref{Conj:abelvarp} is \cref{Conj:p adic analog GPC weak}.
\end{remark}

\section{Tannakian categories of motives}\label{SS:important hom=num or section of NUM}
	We recall here generalities on the tannakian formalism for categories of motives and their fiber functors. 
	We keep notation from \cref{S:conventions}; in particular we work with homological motives with coefficients in $\bQb$ (unless otherwise specified) and the fiber functors take values in vector spaces.
	
	We  work only with motives satisfying  the
	standard conjecture $\sim_\ho = \sim_\num$ in order to have tannakian categories (see \cref{S:conventions}\eqref{hom=num} and \cref{rem:homnum}).

 \begin{notation}\label{def:realization} Consider the following functors and natural transformations (see also \cref{S:conventions}\eqref{rem:cycleclass}).
\begin{align}\label{Eq:realization functors}
\begin{array}{rll}
	&R_\dR \colon \Mot(\bQb)\to \Vect_\bQb			& M\mapsto \rH^*_\dR(M,\bQb)\\
	&R_\crys \colon \Mot(\obF_p)\to \Vect_{\obQp}	& M\mapsto \rH^*_\crys(M,\obQp)\\	
	& R_B \colon \Mot(\bQb)_\bQ\to \Vect_\bQ		& M\mapsto \rH^*_{Betti}(M,\bQ)\\
	&Z_0  \colon  \Mot(\bQb)\to \Vect_\bQb			& M\mapsto \Hom_{\Mot(\bQb)}(\one,M)\\
	&Z_p  \colon  \Mot(\obF_p)\to \Vect_\bQb		& M\mapsto \Hom_{\Mot(\obF_p)}(\one,M)\\
	& \cl_0   \colon  Z_0 \to R_\dR							& f \mapsto R_\dR(f)(1) \\
	& \cl_p   \colon  Z_p \to R_\crys					& f \mapsto R_\crys(f)(1) \\
	\end{array}
\end{align}
where the functor $R_B$ depends on the choice of an embedding $\sigma \colon \bQb \hookrightarrow \bC.$
\end{notation}

\begin{definition} A lax-monoidal functor between two monoidal categories is the data of a functor $\cF\colon (\cC,\otimes)\to (\cD,\otimes)$ together with morphisms
\begin{align*}
	\varepsilon & \colon \one_\cD\to \cF(\one_\cC)\\
	\mu_{X,Y} & \colon \cF(X)\otimes \cF(Y) \to \cF(X\otimes Y), \text{ for all objects } X,Y \text{ of } \cC
\end{align*}
satisfying the usual conditions of associativity and unity.
If moreover $\varepsilon$ and $\mu_{X,Y}$ are isomorphisms for all $X,Y\in\cC$, then we call $\cF$ a strong-monoidal functor (or simply \emph{monoidal}).

A natural transformation between two lax-monoidal functors is required to be compatible with the extra structure. 
\end{definition}

\begin{remark}\label{only lax-monoidal}
The functors $R_B, R_\dR,R_\crys$ are monoidal, whereas $Z_0$ and $Z_p$ are only lax-monoidal because in general, for $X$ and $Y$ varieties, the product $X\times Y$ has more cycles than the product of cycles from $X$ and $Y$.
Moreover the functors $R_B$, $R_\dR$ and $R_\crys$ are exact and faithful, hence they define realization functors (or fiber functors) of their respective source categories turning them into tannakian categories. 
Notice that $R_\dR$ and $R_B$ make the tannakian category neutral whereas $R_\crys$ does not.
On the other hand, $Z_0$ and $Z_p$ are not faithful as there are non zero motives without cycles (e.g. $\fh^1(X)$).
\end{remark}

Recall the notation  $\<\cdot\> $ from \cref{S:conventions}\eqref{not:generation}.
The following is a special case of \cite[Théorème 1.12]{Del-tannak}.

\begin{theorem}\label{T:tann for <M>}
Let $M\in \Mot(\bQb)^\textrm{gr}$ be a motive with good reduction (see  \cref{S:conventions}\eqref{specialization}) and
consider the specialization functor
$\spe\colon \Mot(\bQb)^\textrm{gr} \longrightarrow \Mot(\bFb_p).$
Then the following holds true.
\begin{enumerate}
		\item\label{item:tanna G_dR} There is a reductive algebraic group $G_\dR(M)$ over $\bQb$ and an equivalence of tensor categories induced by $R_\dR$
		\[ \<M\> \simeq \Rep_\bQb(G_\dR(M)). \]
		\item\label{item:tann for <M>:crys factorizes} The functor $R_\crys$ factorizes (non uniquely) through $\Vect_\bQb$ and there exists a reductive algebraic group $G_\crys(M)= G_\crys(\spe(M))$ over $\bQb$ such that $R_\crys$ induces an equivalence of tensor categories:
		\[ \<\spe(M)\> \simeq \Rep_{\bQb}(G_\crys(M)).\]
		\item\label{inclusion groups} The specialization functor together with a choice of factorization of $R_\crys$ as in \eqref{item:tann for <M>:crys factorizes} induce a morphism of algebraic groups $G_\crys(M)\to G_\dR(M)$ exhibiting $G_\crys(M)$ as a closed subgroup of $G_\dR(M)$.
	\end{enumerate}
\end{theorem}
\begin{proof}
The functors $R_\dR$ and $R_\crys$ restrict to fiber functors on $\<M\>$, respectively on $\<\spe (M)\>$. 
Then \eqref{item:tanna G_dR} follows from tannakian reconstruction \cite[Théorème 1.12]{Del-tannak}.
For the second point, we also have that $R_\crys$ is a realization functor but the rings of coefficients do not match. 
In this situation, Deligne shows in \cite[Corollaire 6.20]{Del-tannak} that the realization functor, in our case $R_\crys$, can be factorized through a finite extension of the field of coefficients of the source, hence over $\obQ$. This is highly non canonical since it corresponds to the choice of a closed point in a scheme but it will be of no importance to our purpose.

The reductivity of the groups $G_\dR(M)$ and $G_\crys(M)$ comes from the fact that, since we work with motives satisfying the standard conjecture $\sim_{num}=\sim_{hom}$, Janssen's semisimplicity theorem \cite{Jann} implies that the tannakian category is semisimple.

To prove \eqref{inclusion groups} we will use the criterion from \cite[Proposition 2.21]{Del-Mil-tannak}. We need to check that any object of $\<\spe(M)\>$ is a  subquotient of an object coming from $\<M\>$. This is clear because the functor $\spe$ sends the generator $M$ of the latter to the generator $\spe(M)$ of the former.
\end{proof}
\begin{remark}
As in the classical situation, we could define $\cT_M$, the torsor of $\dR$-$\crys$ <<\emph{periods}>> of $M$, as the functor of points associated to tensor isomorphisms:
\[ \cT_M:=\Isom^\otimes (R_\crys\circ \spe |_{\<M\>}, R_\dR|_{\<M\>}\otimes_{\obQ}\obQ_p).\]
Using  \cref{T:tann for <M>}\eqref{item:tann for <M>:crys factorizes} this variety has a (non canonical) $\bQb$-structure.

The Berthelot isomorphism \eqref{Eq:the Berthelot compar} gives an element $c\in \cT_M(\bQb_p)$. 
However, the position of the point $c\in\cT_M(\bQb_p)$ will depend on this non canonical $\bQb$-structure and therefore it will not have any special properties.

In order to have correct periods, one has to find a second canonical $\bQb$-structure. 
We will see in the sequel that algebraic classes furnish it although several subtleties appear related to the fact that $Z_p$  is not, in general, a fiber functor (see \cref{only lax-monoidal}).
\end{remark}

\begin{remark}\label{rem:homnum} (Technical remark, not needed in the sequel.)
		There are traditionally three techniques to define tannakian categories of motives which allow one to work with varieties for which $\sim_\ho = \sim_\num$ is not known. 
		All of them are perfectly fine for the study of classical complex periods but none of them is suitable to our context. 
		
		The first is the construction of motives avoiding algebraic cycles due to Nori. 
		The problem here, among other things, is that we do not know yet how to express algebraic classes in Nori's category of motives. This presents an impediment to finding a $\bQb$-structure to crystalline cohomology.
		
		The second is Andr\'e's category of motivated cycles. 
		Here the field of coefficients of the category, which is conjecturally $\bQ$, may a priori be  too large (for example, it might  contain  transcendental numbers).
		
		 Finally, one could work with numerical motives and use André--Kahn's section  \cite[Section 9]{Andmot}).
		 In this case it is not clear if one can define the specialization map $\spe\colon \Motgr \rightarrow
	\Mot(\bFb_p)$.
		 Also, the analogous   of \cref{Conj:liftp} would not be reasonable. For example, let $X$ be a hypothetical variety which verifies  $\sim_\ho = \sim_\num$ in positive characteristic  but not in characteristic zero. We have the following diagram
		 \[ \begin{tikzcd}
		 	Z_0^{\hom}(X_0) \ar[r,"\alpha_0", twoheadrightarrow]\ar[d,"\spe", hook] & Z_0^\num(X_0) \arrow[l, bend right =40, "s_0"]\\
		 	Z_p^{\hom}(X_p) \ar[r, "\alpha_p", "\sim"'] & Z_p^\num(X_p) \arrow[l,bend left=40, "s_p"']
		 	\end{tikzcd}\]		 
		 where $s_0$ and $s_p$ are the André--Kahn sections.
		 The inclusion 
		\[s_0(Z_0^\num(X_0)) \subset s_p(Z_p^\num(X_p))\cap H_\dR(X/\obQ)\] 
		is strict as the latter contains at least  $Z_0^{\hom}(X_0)$.
		 \end{remark}

\section{Andr\'e's $p$-adic periods}\label{SS:periods}
We define in this section a class of $p$-adic periods (not related to   Fontaine's ring  $B_\crys$) 
 which form a countable subring of $\obQ_p$. 
We state a bound (whose proof is given in \cref{T:trdeg <= dim H_M}) on the transcendence degree of the periods of a given motive and spell out some examples.

Throughout the section $M\in \Motgr$ is a motive with good reduction (see \cref{S:conventions}\eqref{specialization}), $G_\crys(M)$ and $G_\dR(M)$ are the tannakian groups constructed in 
\cref{T:tann for <M>} and  \[Z_p(\spe(M))=\Hom_{\Mot(\obF_p)}(\one,\spe(M))\]
is the $\bQb$-vector space of algebraic classes modulo $p$ (\cref{def:realization}).

\begin{definition}\label{D:matrix of periods of M}
Fix a basis $\mathcal{B}$ of the $\bQb$-vector space $Z_p(\spe(M))$ and a basis $\mathcal{B}'$  of the $\bQb$-vector space $R_\dR(M)$.

Identify $\mathcal{B}$ with a list of linearly independent vectors in  $R_\crys(\spe(M))$ (via $\cl_p(\spe(M))$; see \cref{def:realization}) and $\mathcal{B}'$ with a basis of the  $\bQb_p$-vector space $R_\crys(\spe(M))$ (through Berthelot's comparison isomorphism \eqref{Eq:the Berthelot compar}).

Define $\Mat(M)= \Mat_{\mathcal{B},\mathcal{B}'}(M)$, the matrix of	Andr\'e's $p$-adic periods of $M$ (with respect to the bases $\mathcal{B}$ and $\mathcal{B}'$), as the (vertical rectangular) matrix with coefficients in $\bQb_p$ containing the coordinates of $\mathcal{B}$ with respect to $\mathcal{B}'$ (written in columns).\end{definition}

\begin{remark}\label{rem:rectangular}
The elements of $\mathcal{B}$   are linearly independent in $R_\crys(\spe(M))$  because homological and numerical equivalence coincide on $M$ by hypothesis (see \cref{SS:important hom=num or section of NUM}).
In general, they do not span the whole space, this is why the matrix is only rectangular (contrary to the classical case of complex periods where it is a square matrix).

The matrix  $\Mat_{\mathcal{B},\mathcal{B}'}(M)$ depends on the chosen bases $\mathcal{B}$ and $\mathcal{B}'$ but its $\bQb$-span inside $\bQb_p$ does not.
\end{remark}

\begin{definition}\label{period triples}
Fix a motive $M\in \Motgr$, a vector $v \in Z_p(\spe(M))$ and a covector $\xi \in R_\dR(M)^{\vee}$. To such a triple we associate a number 
\[[M,v,\xi] \in \bQb_p,\] 
called the $p$-adic period associated with the triple, defined as
$[M,v,\xi]= \xi(v)$, where we see $v$ as a vector in $R_\crys(\spe(M))$ via $\cl_p(\spe(M))$ and $\xi$  as a vector in $R_\crys(\spe(M))^{\vee}$ through Berthelot's comparison isomorphism.
 \end{definition}

\begin{remark} 
All coordinates of the matrix $\Mat_{\mathcal{B},\mathcal{B}'}(M)$  are examples of such $[M,v,\xi]$. Conversely, as soon as $v$ and $\xi$ are non zero, one can find bases $\mathcal{B}$ and $\mathcal{B}' $  such that $[M,v,\xi] $  appears as an entry in $\Mat_{\mathcal{B},\mathcal{B}'}(M)$.
\end{remark}

\begin{definition}\label{D:algebra of periods of M}
Consider $\Mat(M)$ as in \cref{D:matrix of periods of M} and define $\cP_p(M)\subset \bQb_p$ as the smallest $\bQb$-algebra in $\bQb_p$ containing the coordinates of the matrix $\Mat(M)$. 

Equivalently, $\cP_p(M) $ is the   $\bQb$-subalgebra of $\bQb_p$ generated by the numbers $[M,v,\xi] \in \bQb_p$ for all possible choices of $v$ and $\xi$. 

The algebra $\cP_p(M)$ is called the algebra of Andr\'e's $p$-adic periods of $M$. 
Any element in this algebra will be called a  $p$-adic period of $M$.
\end{definition}
 
\begin{remark}\label{r:p-adic contains Frobenius periods}($p$-adic periods contain coefficients of Frobenius) 
	Some authors usually call $p$-adic periods the coefficients of the Frobenius matrix \cite{Furusho,Brownp} acting on the de Rham cohomology.
	Let us explain that André's $p$-adic periods include these ones. 
	
	Consider a variety $X$ over $\bQ$ (or a number field) with good reduction $X_p$ at $p$.
	Through Berthelot's comparison isomorphism \eqref{BO intro}, one can consider the Frobenius of $X_p$ acting on $\rH^i_\dR(X/\bQ)\otimes\bQ_p$ and take its matrix with respect to a basis coming from $\rH^i_\dR(X/\bQ)$ for various $i\ge 1$.
	The coefficients of this matrix, or the $\bQ$-algebra they generate, are called the $p$-adic periods of $X$ in \cite{Furusho,Brownp}. Let us call them \emph{Frobenius periods}.
	
	In order to see them as special cases of André's $p$-adic periods, consider for various $i\ge1$ the internal End motive
	$M:=\cEnd(\fh^i(X),\fh^i(X))$. 
	Then the Frobenius of $X_p$ gives a natural element $\gamma_{\Frob}$ in $Z_p(M)$. Notice that the realization $R_\dR(M)$ coincides with $ \End(\rH^i_{\dR}(X/\bQ))$ and the coefficients of $\gamma_{\Frob}$ in a basis of $R_\dR(M)$ are precisely the Frobenius periods that we introduced above.
	
	In general, as we point out in \cref{more than frob}, the motive $M_p$ could have many more cycles than the Frobenius endomorphism and all of them give $p$-adic periods as we defined them. 
\end{remark}

\begin{remark}\label{tensor periods}
Contrary to the classical situation of complex periods, it is not true that the $p$-adic periods of $M$ contain the $p$-adic periods of every object $N\in\<M\>$. 

For singular and de Rham cohomology, all classes on $N$ are polynomials on classes on $M$ due to the Künneth formula. A slightly different way of saying this is that taking singular cohomology is a fiber functor.
This is false for algebraic classes and intimately related to \cref{only lax-monoidal}. Similarly, one can see this as the failure of algebraic classes to provide a fiber functor. 

For instance, it could happen that $M$ has no algebraic cycles whereas the motive $M\otimes M^\vee \in\<M\>$ always has at least one algebraic class, namely the identity of $M$.
But even if a motive has algebraic classes, it is not true in general that  they generate all the algebraic classes of objects in $\<M\>$ (see \cref{eg:sym4 no Kunnet for periods}).
\end{remark}

\begin{definition}\label{D:field of p-adic periods of <M>}
We define $\cP_p(\<M\>)\subset \obQ_p$ to be the $\bQ$-subalgebra generated by the $p$-adic periods of all objects $N\in\<M\>$.
It is called the algebra of Andr\'e's  $p$-adic periods of $\<M\>$ or simply the algebra of    $p$-adic periods of $\<M\>$.
\end{definition}
\begin{remark}
	Clearly the algebra $\cP_p(M)$ is finitely generated by construction. 
	It is shown in \cref{r:surj} that the bigger algebra $\cP_p(\<M\>)$ is also finitely generated.
	It follows that the $\obQ$-subalgebra of $\obQ_p$ generated by all $\cP_p(\<M\>)$, with $M$ varying in $\Mot(\obQ)$, is countable. Hence not all $p$-adic numbers are periods.
	\end{remark}

\begin{theorem}\label{T:trdeg <= dim group}
	We have the inequalities
	\begin{align}\label{Eq:trdeg<=dim group}
	\trdeg_\bQb(\cP_p(M))\le 	 \trdeg_\bQb(\cP_p(\<M\>))\le  \dim G_\dR(M)  -  \dim G_\crys(M).
	\end{align}
\end{theorem}
\begin{proof}
	The first inequality follows from the inclusion $\cP_p(M)\subset \cP_p(\<M\>)$. For the second use \cref{T:cH_M is repr} and  \cref{T:trdeg <= dim H_M}.
\end{proof}
\begin{remark}\label{rem:bound rectangular}
In the classical case of complex periods  $\cP_\bC(M)$ one has the inequality       $\trdeg_\bQb(\cP_\bC(M))\le  \dim G_\dR(M) $. 
(This inequality is expected to be an equality by the Grothendieck period conjecture.) 
In \cref{T:trdeg <= dim group}, the upper bound  is instead $\dim G_\dR(M)  -  \dim G_\crys(M)$.
 One may interpret this difference as related to the fact that the $p$-adic period matrix is only rectangular (\cref{rem:rectangular}).
 \end{remark}
 \begin{remark}\label{rem:bound rectangular}
As in the classical case, the dimensions of $G_\dR(M)$ and $ G_\crys(M)$ are sometimes computable because they are (expected to be) related to tannakian groups of purely cohomological nature.

Indeed, if $M$ comes from a motive with coefficients in $\bQ$, then the Hodge conjecture predicts the equality $\dim G_\dR(M)  = \dim \MT(M) $, where $\MT(M)$ is the Mumford--Tate\footnote{The Tannakian group of the Hodge structure $R_B(M)$.} group of $M$.
Similarly, the Tate conjecture predicts the equality $ \dim G_\crys(M)=  \dim (\varphi \vert R_\crys(M))$, where   $(\varphi \vert R_\crys(M))\le \GL(R_\crys(M)$ is the Zariski closure of the subgroup generated by the Frobenius morphism $\varphi$ acting on $R_\crys(M))$, for a suitable model of $M$ defined over a finite field.
(Note that the groups $ \MT(M) $ and $(\varphi \vert R_\crys(M))$ are often computable.)

These conjectures are widely open, but known in some already interesting cases, for example products of elliptic curves (see \cite{Imai} over $\bC$ and \cite{Spiess} over finite fields). 
 \end{remark}

Let us give some examples. 

In all but one case, $E$ is an elliptic curve over $\obQ$ with good reduction $E_p$ at $p$ and $M = \cEnd(\fh^1(E))$.
Notice that in this situation a cycle $\gamma\in Z_p(M_p)$ is the same thing as an endomorphism $f$ of $E_p$ and its $p$-adic periods are precisely the coefficients of $f$ with respect to a basis of $\rH^1_\dR(E/\bQb)$.

\begin{example}\label{Eg:H_M for E=CM}
	Let $E/\bQb$ be a CM elliptic curve and denote by $E_p$ its reduction modulo $p$.
	We denote by $L$ its CM field. The motive $\fh(E)\in \Mot(\obQ)$ of $E$ decomposes as $\fh(E) = \fh^0(E)\oplus\fh^1(E)\oplus\fh^2(E)$.

	It is known that $G_\dR(\mathfrak{h}^1(E))$ is isomorphic to $\bG_m^2$.  
	Essentially, this amounts to checking that the commutator of the CM field $L$ inside $Mat_2(\bQb)$ is $L\otimes_\bQ \obQ\simeq \obQ\times \obQ$.
	Choosing a basis of $\rH^1_\dR(E)$ provides us with an equivalence $\<\mathfrak{h}^1(E)\>\simeq \Rep(\bG_m^2)$ and this can be done in such a way that $\mathfrak{h}^1(E)$ corresponds to the representation $\obQ(1,0)\oplus \obQ(0,1)$.
	
	Under the chosen equivalence, the motive $M:=\cEnd(\fh(E))=\cEnd(\mathfrak{h}^1(E))$ corresponds to the representation $\obQ(1,-1)\oplus \obQ(-1,1)\oplus \obQ(0,0)^{\oplus 2}$ of $\bG_m^2$.
	In particular, it is not a faithful representation and the subcategory of motives that it generates $\<M\>\subset \< \mathfrak{h}^1(E) \>$ is equivalent to $\Rep(\bG_m^2/\Delta\bG_m)$, i.e., to $\Rep(\bG_m)$.
	This shows that $G_\dR(M)\simeq\bG_m$.

	Now, if $E_p$ is not supersingular (i.e., when $p$ splits into two distinct prime ideals in the CM field $L$), the same argument as above gives $G_\crys(M)=\bG_m$.  Hence the bound in  \cref{T:trdeg <= dim group} implies that all the $p$-adic periods are algebraic. One can also see this directly from the fact that all algebraic classes lift to characteristic zero. Moreover, if the elliptic curve is defined over $\bQ$, then these algebraic numbers are contained in any field extension of $\bQ$ over which $E$ acquires CM.
		
	 The interesting case is when $E_p$ is supersingular. Then we have $\spe(M)\simeq \one^{\oplus 4}$, hence $G_\crys(M)$ is the trivial group. 
	By \cref{T:trdeg <= dim group} we have $\trdeg(\cP_p(\<M\>))\le 1$.

	Let us see this in a more classical way.
	The CM field $L$ is generated by $f\in\End(E)$ which can be chosen such that $f^2$ is a rational multiple of the identity.
	The division algebra $D:=\End(E_p)$ contains $L$ and becomes a vector space of dimension $2$ over it. Denote by $\ad_f\colon D\to D, a\mapsto faf\inv$ the conjugation by $f$.
	We have $\ad_f^2=\id$ and so the eigenspace decomposition of $D$ is of the form  $D=L\id\oplus Ld$ where $d$ satisfies $fd=-df$.

	The $p$-adic periods of $M$ are generated by the matrix of the action of $d$ on $\rH^1_\crys(E_p,\obQ_p)$ written in a basis of $\rH^1_\dR(E/\obQ)$. 
	This is a $2\times 2$ matrix and we need to find $3$ relations between the $4$ coefficients.
	These are given by the determinant and the trace of $d$, which are rational numbers and the relation $fd=-df$.
	Altogether this shows that $\trdeg(\cP(\<M\>))\le 4-3=1$.
 
	The $p$-adic periods appearing in this example are products of special values of the $p$-adic Gamma function.
	As we explained above, only one entry matters in the $p$-adic period matrix.
	Let us describe such an entry following 
	\cite[I.4.6]{Andre-p} which is based on \cite{Ogus,Colem}. 

	We start by introducing some notation. Let $d\in\bN$ be the fundamental discriminant  such that the elliptic curve $E$ has CM by $\bQ(\sqrt{-d})$.
	The embedding $\bQ(\sqrt{-d})\subset \bQ(\zeta_d)$ induces a map of Galois groups 	$(\bZ/d)^\times\to\bZ/2$ which we denote by $\varepsilon$.
	Let $h$ be the class number of $\bQ(\sqrt{-d})$ and $w$ the number of roots of unity in $\bQ(\sqrt{-d})$.
	Finally, we consider the function
	\[ \left\langle\frac{.}{d}\right\rangle\colon (\bZ/d)^\times \to (0,1]\cap \bQ \text{ characterized by } d\left\langle\frac{u}{d}\right\rangle \equiv u.\]
	
	\textbf{Unramified case.} We assume  that $p$ does not ramify in $\bQ(\sqrt{-d})$. Under this hypothesis, the significant entry of the $p$-adic periods matrix is 
	\[ \prod_{u\in (\bZ/d)^\times} \Gamma_p\left(\left\langle p\frac ud\right\rangle\right)^{-\varepsilon(u)w/4h}.\]
	
	As we explained above, in the ordinary case the $p$-adic periods are algebraic numbers. This can also be seen in the above explicit description using the Gross--Koblitz formula.
	In the supersingular case we expect this number to be transcendental (see \cref{Conj:p adic analog GPC}).
	
	\medskip
	\textbf{Ramified case.} If $p$ ramifies in $\bQ(\sqrt{-d})$, the curve $E_p$ is still supersingular and there is a similar, but more complicated, formula worked out by Coleman \cite[6.5]{Colem}. 
	We give here some simplification of it in a special case following André (\cite[I.4.6.4]{Andre-p}).
	
	Let $p=3$ and $d=3n$ for $n\in \bN$ coprime with $p$.  Then the relevant $p$-adic period is of the form
	\[ \kappa:=\prod_{u\in (\bZ/n)^\times} \Gamma_p\left(\left\langle\frac un\right\rangle\right)^{\left(\frac nu\right)w/2h} \]
	where $ \left(\frac n{u}\right)$ is the Jacobi symbol.
	
	If $(\frac np)=1$ André shows that $\kappa\in\obQ$ as a consequence of the Gross--Koblitz formula.
	He furthermore points out that for $n=8$ one has $(\frac np)=-1$ and $\kappa$ is still algebraic. This time the argument is based on the functional equation of $\Gamma_p$.
\end{example}

\begin{example}\label{Eg:H_M for E=ordinary}
	Let $E/\bQb$ be a non CM elliptic curve and consider the endomorphism motive $M:=\cEnd(\mathfrak{h}^1(E))$.
	It is known that the tannakian group $G_\dR(\mathfrak{h}^1(E))$ is $\GL(\rH^1_\dR(M))$ and so isomorphic to $\GL(2)$. We can fix an equivalence of categories $\<\mathfrak{h}^1(E) \>\simeq \Rep(\GL(2))$ such that the motive $\fh^1(E)$ corresponds to the natural representation.
	
	The motive $M$  corresponds to the adjoint representation of $\GL(2)$ on the Lie algebra $\fgl(2) = \End(\obQ^2)$.
	Since the action of $\GL(2)$ on $\fgl(2)$ factors through a faithful representation of $\PGL(2)$, we deduce that the subcategory generated by $M$ is equivalent to $\Rep(\PGL(2))$.
	(Here we have used that for an algebraic group, any faithful representation tensor-generates the whole category of representations of this group; see for example \cite[\S3.5 Theorem]{Wat}.)
	
	As a consequence, we get $G_\dR(M)\simeq \PGL(2)$.
	Now we consider two cases: 
	\begin{itemize}
		\item $E_p$ supersingular. By definition, we have $\spe(M)\simeq \one^{\oplus 4}$
		and we therefore get $G_\crys(M)=1$.
		On the other hand we have $\cP_p(M)=\cP_p(\<M\>)$, the reason being that matrix coefficients for a faithful representation of a reductive group generate the matrix coefficients of all representations.
		In conclusion
		\cref{T:trdeg <= dim group} gives us the bound $\trdeg(\cP_p(M))\le 3$.
		
		The same bound can be deduced in an elementary way similarly to the argument from  \cref{Eg:H_M for E=CM}.
		The division algebra $D:=\End(E_p)$ is generated by two elements, say $f$ and $g$, satisfying the relations $fg=-gf$. Each of them has four matrix coefficients. The trace and the determinant give two conditions per endomorphism. Moreover, the above relation $fg=-gf$ gives a fifth, independent condition. 
		In total we therefore have that the transcendence degree is bounded by $4+4-2-2-1=3$. 
		We conjecture that there are no further algebraic relations (see \cref{Conj:p adic analog GPC}).		
		
		\item $E_p$ not supersingular. Then the same argument as in \cref{Eg:H_M for E=CM} proves that $G_\crys(M)\simeq \bG_m$.
		We can see it as the maximal torus $T$ in $\PGL(2)$. 
		\cref{T:trdeg <= dim group} gives the bound $\trdeg(\cP_p(\<M\>))\le 2$.
		We conjecture that this should be an equality (see \cref{Conj:p adic analog GPC}).
		 
		\noindent Here as well, we  have the equality $\cP_p(M) = \cP_p(\<M\>)$; see \cref{Eg:H_M for E=ordinary new} for more details.
	\end{itemize}
\end{example}

\begin{example}
Let $X/\bFb_p$ be  a smooth projective  variety and consider two lifts $\tilde{X}/\bQb$  and $\tilde{X'}/\bQb$.
The Berthelot  comparison theorem \eqref{Eq:the Berthelot compar} gives
\[\rH^*_{\dR}(\tilde{X}/\bQb)\otimes\obQ_p \simeq \rH^*_\crys(X,\obQ_p) \simeq \rH^*_{\dR}(\tilde{X'}/\bQb)\otimes\obQ_p. \]
Through this identification one can write the coordinates of vectors in $\rH^*_{\dR}(\tilde{X}/\bQb)$ with respect to a  $\bQb$-basis of $\rH^*_{\dR}(\tilde{X'}/\bQb)$.

These coordinates are special cases of the $p$-adic periods we defined. In order to see this consider the motive $M:=\cHom(\mathfrak{h}(\tilde{X}),\mathfrak{h}(\tilde{X'}))$ and notice that $\spe(M) = \cHom(\fh(X),\fh(X))$ and hence it contains the identity morphism. In \cref{D:matrix of periods of M} one can choose a basis $\mathcal{B}$ of $Z_p(\spe(M))$ containing, among its elements, the identity morphism. The coefficients of this vector in a basis of $R_\dR(M)$ are the aforementioned coordinates.
Below we make this more explicit for elliptic curves.
\end{example}
	
\begin{example}\label{Eg:two ell curves}
	
As a special case of the above example, we will consider the Legendre family of elliptic curves where the $p$-adic periods are computed by Katz in \cite[IV]{Katz}. 

Fix a prime $p$ and let $\cE\colon y^2 = x(x-1)(x-\lambda)$ be the Legendre family over $\bA^1$. 
Consider $\lambda_0$ such that $E:=\cE_{\lambda_0}$ has complex multiplication and its reduction $E_p$ modulo $p$ is ordinary. 
For $\lambda$ sufficiently close $p$-adically to $\lambda_0$, we denote by $E(\lambda)$ the fiber of $\cE$ at $\lambda$. 
Katz describes the $p$-adic periods appearing  in the identification (through Berthelot's comparison isomorphism) 
\[\rH_\dR^1(E)\otimes\bQ_p\simeq \rH_\dR^1(E(\lambda))\otimes\bQ_p\]
by interpreting it as the identification induced by the horizontal solutions of the Gauss--Manin connection.

To do so, consider $\omega = dx/2y\in \rH^1_{\dR}(E(\lambda))$ and $\nabla$ the connection on the de Rham vector bundle.
We have $\nabla\omega = (x-\lambda)dx/2y$ and $(\omega$, $\nabla\omega)$ form a basis of $\rH^1_\dR(E(\lambda))$ for all $\lambda$.
For $\lambda = \lambda_0$ we choose a basis of $\rH^1_\dR(E)$ formed by eigenvectors for the CM action.
This can be chosen to be of the form $u = \omega|_{\lambda_0}$ and $v=(\nabla\omega-e\omega)|_{\lambda_0}$ for some algebraic constant $e$.

Consider the hypergeometric equation
\begin{equation}\label{eq:hypergeom_equat}
	 D^2f = -\lambda(\lambda-1)f 
\end{equation}
where $D=\lambda(\lambda-1)d/d\lambda$. 
Let $\alpha$, $\beta$ be the solutions normalized by
\[ \begin{array}{ll}
	\alpha(\lambda_0) = 1, & D(\alpha)(\lambda_0) = e\\
	\beta(\lambda_0) = 0, & D(\beta)(\lambda_0) = 1
\end{array}.\]
Then the horizontal translate of the basis $(u,v)$ of $\rH_\dR^1(E)$ to $\rH_\dR^1(E(\lambda))$ is
\[ \begin{array}{l}
	U = D(\beta)\omega-\beta\nabla\omega\\
	V = -D(\alpha)\omega+\alpha\nabla\omega.
\end{array} \]
which gives us the sought for $p$-adic periods as the coefficients of the matrix
\begin{equation}\label{eq:matrix_hyperg}
 \begin{pmatrix}
	D(\beta)(\lambda) & -D(\alpha)(\lambda)\\
	-\beta(\lambda) & \alpha(\lambda)
\end{pmatrix}
\end{equation}

 Let us compute the tannakian groups for the motive $M(\lambda) = \cHom(E, E(\lambda))$.
 First, we have   $G_\crys(M(\lambda)) = \bG_m$ as $\spe(M(\lambda))$ is simply $\cEnd(E_p)$ whose tannakian group has been computed in \cref{Eg:H_M for E=ordinary}.
 
Second, let us show that $G_\dR(M(\lambda)) = \GL(2)$. 
Let $G(\lambda)$ be the tannakian group of the category $\<\fh^1(E),\fh^1(E(\lambda))\>$.
Since the determinants of $\fh^1(E)$ and $\fh^1(E(\lambda))$ are isomorphic, we deduce that $G(\lambda)$ fits in the pullback diagram:
\[ \begin{tikzcd}
G(\lambda) \arrow[d,twoheadrightarrow] \arrow[r,twoheadrightarrow]& \bG_m^2 \ar[d,"\det"]\\
\GL(2) \ar[r,"\det"] & \bG_m.
\end{tikzcd}
 \]	
Explicitly, we have 
\[ G(\lambda) = \{(a,b,A)\mid ab=\det(A)\} < \bG_m^2\times \GL(2).\]

The tannakian group $G_\dR(M(\lambda))$ is the quotient of $G(\lambda)$ by the diagonal subgroup $\bG_m=\{(a, a, aI_2)\} \le  \bG_m^2\times \GL(2)$ and as such it is isomorphic to $\GL(2)$.
One can also check that the image of $G_\crys(M(\lambda))$ in $G_\dR(M(\lambda))\simeq \GL(2)$ is given by $t\mapsto \mathrm{diag}(1,t)$.

Notice that $Z_p(M(\lambda))$ is generated by $\id$ and $\Frob$ and $\dim\rH^1_\dR(M(\lambda))=4$.
Hence the $p$-adic period matrix is of size $4\times 2$.
The first column is given by the four entries in the matrix \eqref{eq:matrix_hyperg}.
As $\Frob$ lifts to an endomorphism of $E$, we have that the coefficients of the second column are a $\obQ$-linear combination of the above four entries.

In conclusion, \cref{T:trdeg <= dim group} gives us the bound $\trdeg(\cP_p(M(\lambda)))\le 3$.
It follows that the coefficients of the matrix \eqref{eq:matrix_hyperg} satisfy one algebraic relation.
An explicit one can be given using the Wronskian of the solutions of the hypergeometric equation \eqref{eq:hypergeom_equat}. 
This is a reflection of the tannakian constraints defining $G(\lambda)$.
We conjecture that there are no further algebraic relations (see \cref{Conj:p adic analog GPC}).
	\end{example}

\begin{remark} (Frobenius matrix vs $p$-adic periods)\label{more than frob} 
	We showed in \cref{r:p-adic contains Frobenius periods} that André's $p$-adic periods contain the coefficients of Frobenius. Let us further argue that this inclusion should be strict.
	
	Consider $M = \cEnd(\fh^1(E))$ where $E$ is a non CM elliptic curve with supersingular reduction $E_p$ as in the above example. 
	Recall that a cycle $\gamma\in Z_p(M_p)$ is the same thing as an endomorphism $f$ of $E_p$ and its $p$-adic periods are precisely the coefficients of $f$ with respect to a basis of $\rH^1_\dR(E/\bQb)$. 
	Arguing as in \cref{Eg:H_M for E=CM}, we have that the transcendence degree of its periods is at most $2$. In particular, the coefficients of the Frobenius endomorphism have transcendence degree at most 2.
	On the other hand, \cref{Conj:p adic analog GPC} predicts that the inequality $\trdeg(\cP_p(M))\le 3$ is actually an equality. 
	In particular, there should be more $p$-adic periods than just the coefficients of the Frobenius matrix.
	
	See also \cref{remark ultima} for a detailed comparison between Frobenius coefficients and André's $p$-adic periods.
\end{remark}

\section{The homogeneous space of periods}\label{SS:homogeneous}
In this section we establish the tannakian framework for studying $p$-adic periods. As a consequence, we will prove \cref{T:trdeg <= dim group}.

Recall that in the classical situation of comparing de Rham and singular cohomology of a smooth projective variety $X$ defined over a number field, one defines  a torsor $\cP_X$ called the torsor of periods of $X$. 
It is a torsor over the de Rham tannakian group as well as over the Betti tannakian group.
The comparison isomorphism given by integration defines a $\bC$-point of this torsor.

In the mixed characteristic situation, for a smooth projective scheme $X$ over $\cO_{\fp}$, we cannot 
define a torsor but we can define a homogeneous space $\cH_X$ for the de Rham tannakian group
 that we will call the homogeneous space of $p$-adic periods of $X$.
The Berthelot comparison isomorphism between de Rham and crystalline cohomologies provides a $\obQp$-point of $\cH_X$.

Throughout the section $M\in \Motgr$ is a motive with good reduction (see \cref{S:conventions}\eqref{specialization}) and  $\<M\>$ is  the category it generates (see Conventions in  \cref{S:conventions} \eqref{not:generation}).
We will make use of the reductive  tannakian groups $G_\crys(M)$, $G_\dR(M)$  constructed in 
\cref{T:tann for <M>} and of the functors  $Z_p$, $\cl_p,R_\crys$ and $R_\dR$ from \cref{def:realization}.

\begin{definition}\label{D:of H_M} 
	We define the following set-valued functors on the category of $\obQ$-algebras 
\begin{align}
	 \cH_M&:=\Emb^\otimes(Z_p\circ \spe|_{\<M\>}, R_\dR|_{\<M\> })  \subseteq	\cH'_M:= \Nat^\otimes (Z_p\circ \spe|_{\<M\>}, R_\dR|_{\<M\> })
\end{align}
where $\Nat^\otimes$ stands for lax-monoidal natural transformations and
$\Emb^\otimes$ stands for those that are injective.

We call $\cH_M$ the \emph{space of $p$-adic periods of the category $\<M\>$}. 
\end{definition}
\begin{remark}\leavevmode
	\begin{enumerate}
		\item  In the above definition we did not use any $\bQb$-structure of crystalline cohomology, for example the one coming from 		\cref{T:tann for <M>}\eqref{item:tann for <M>:crys factorizes}, hence this object is canonically associated to $M$.
		\item The spaces $\cH_M$ and $\cH'_M$ are endowed with a natural action of $G_\dR(M)$.  	\end{enumerate}	
\end{remark}

\begin{theorem}\label{T:cH_M is repr}
The inclusion $\cH_M\subseteq	\cH'_M$
is actually an equality 
\[\cH_M =	\cH'_M.\]
Moreover, the space $\cH_M $ is representable by an affine variety of finite type over $\bQb$ (still denoted $ \cH_M$) which is a homogeneous space under $G_\dR(M)$.

Finally, we have 
  \[\dim \cH_M =  \dim G_\dR(M)-\dim G_\crys(M).\] 
\end{theorem}
A more precise (but non canonical) version of this statement is the following.
\begin{theorem}\label{T:cH_M is repr precise}
Fix a (non canonical) factorisation of 
$R_\crys$ through $\bQb$ as  in \cref{T:tann for <M>}\eqref{item:tann for <M>:crys factorizes} and consider the induced inclusion $G_\crys(M) \subset G_\dR(M)$ from \cref{T:tann for <M>}\eqref{inclusion groups}. Then the spaces $\cH_M $ and $	\cH'_M$ are both representable by the affine variety
$G_\dR(M)/G_\crys(M)$. Moreover the identification 
\[\cH_M =	\cH'_M = G_\dR(M)/G_\crys(M)\]
is compatible with the action of $G_\dR(M)$.
\end{theorem}
\begin{proof}

In light of \cref{T:tann for <M>}, the functor $Z_p$ corresponds to taking invariants
\begin{align*}
		 \Rep_\bQb(G_{\dR}(M))\to \Vect_\bQb \\ 
		 V\mapsto V^{G_\crys(M)}.
\end{align*}
Moreover, the space $\cH'_M$ can now be described as
\begin{align*}
	\cH'_M= \Nat^\otimes ((-)^{G_\crys(M)},\Forg)
\end{align*}
where $\Forg\colon \Rep_\bQb(G_\dR(M))\to\Vect_\bQb$ is the forgetful functor.  

In this context we can use the following theorem of T. Haines. 
To state it, let us  recall that if $G$ is a linear algebraic group and $H$ is a subgroup of $G$, then the affine closure of $G/H$ is defined to be $\overline{G/H}:=\Spec(\cO(G/H))$. 
 It comes equipped with a natural action of $G$ and a $G$-equivariant map $G\to \overline{G/H}$.
 
\begin{theorem}[{\cite[Proposition 2.3, Claim 2.4, Lemma 5.2]{Hai}, \cite{HaiJin}}]\label{T:Haines}
	Let $G$ be a linear algebraic group over an algebraically closed field $k$ and let $H\le G$ be a closed subgroup.
	The space of lax-monoidal natural transformations $\Nat^\otimes((-)^H,\Forg)$ between the two functors $(-)^H,\Forg\colon \Rep_k(G)\to\Vect_k$ is representable by $\overline{G/H}=\Spec(\cO(G/H))$, the affine closure of $G/H$.	
	
	Moreover, the space of embeddings $\Emb^\otimes((-)^H,\Forg)$ corresponds to the image of the canonical map $G\to \overline{G/H}$. 
\end{theorem}

 This result implies that the space $\cH'_M$ is representable by   the affine closure 
	 $\overline{G_\dR(M)/G_\crys(M)} $ 
	and that the space $\cH_M$ is representable by the image of $G_\dR(M)$ in it.

On the other hand, the linear group $G_\crys(M)$ is reductive hence the quotient $G_\dR(M)/G_\crys(M)$ is already affine. In particular it coincides with its affine closure, hence   $\cH'_M= G_\dR(M)/G_\crys(M)$. 
Moreover, the image of $G_\dR(M)$ in this quotient is the whole space, so that $\cH_M= G_\dR(M)/G_\crys(M)$ as well.
\end{proof}

\begin{remark}
	In order to prove representability of $\cH_M$ 
	we needed to factorize $R_\crys$ through $\bQb$ (as in \cref{T:tann for <M>}\eqref{item:tann for <M>:crys factorizes}).
	It would be nice to bypass it and prove the representability directly.
\end{remark}
If one starts with a motive defined over a subfield $K\subset \bQb$, one might be interested to know whether also the space of periods can be defined over $K$. 
In the realm of classical periods, this is true and it has important applications (e.g., for multiple zeta values).
In the $p$-adic context we show that this is also possible. For completeness, we sketch a proof of this result here but it will not be used in the rest of the paper.

\begin{corollary}
	Let $M\in\Mot(K)^{\gr}_K$ be a motive with good reduction defined over $K\subset \bQb$ with coefficients in $K$.
	Define $\cH_M$ as in \eqref{D:of H_M}. Then $\cH_M$ is representable by a variety over $K$.
\end{corollary}
\begin{proof}	
	In the proof of \cref{T:cH_M is repr precise} we never used the fact that $M$ was a motive defined over $\bQb$, only that the coefficients were $\bQb$.\footnote{Beware though that the groups $G_\dR$ and $G_\crys$ are sensitive to the field of definition and so is the space of periods $\cH_M$ (think of a CM elliptic curve defined over $\bQ$ where the Galois group is either a non split torus or its normalizer; see \cite[7.6.4.1]{Andmot}).} 
	In particular, denoting by $\overline M$ the motive $M$ viewed as a motive with coefficients in $\bQb$, we have a variety  	\[ \cH_{\overline M}:=\Emb^\otimes(Z_p\circ\spe |_{\<\overline M\>}, R_\dR|_{\<\overline M\>}) \]
 over $ \Spec(\bQb)$. 
	The goal is to show that this variety descends to $K$.
		
	Consider the following functor of points
	\[ \cH_{M,K}:=\Emb^\otimes(Z_p\circ\spe |_{\< M\>}, R_\dR|_{\< M\>}).\]
	There is a natural transformation
	\[\cH_{M,K}\times_K \Spec(\bQb) \to \cH_{\overline M}.\]
	It can be  checked that this is an isomorphism when evaluated at any $\bQb$-algebra using the two basic properties of the extension of scalars recalled in
	\cref{S:conventions}\eqref{extending scalars}.
	
	By Galois descent we can conclude that $\cH_{M,K}$ is also representable by a variety over $K$.
\end{proof}

\begin{definition}\label{BO point}
 We define the Berthelot point
\[ \fcB\in \cH_M(\bQb_p)\]
as the point corresponding to the following natural transformation
\[\cl_p(\spe(\cdot)) \colon Z_p(\spe(\cdot))\otimes \bQb_p \longrightarrow R_\crys(\spe(\cdot))\simeq R_\dR(\cdot)\otimes \bQb_p \] where the last identification is the Berthelot comparison isomorphism \eqref{Eq:the Berthelot compar}.

We define the algebra of motivic André's $p$-adic periods of $\<M\>$ as 
\[\cAp(\<M\>):=\cO(\cH_M).\]
Evaluating functions at the point $\fcB\in\cH_M(\bQb_p)$ gives a map of algebras 
	\[ \eval_{B}\colon \cAp(\<M\>)\to \obQ_p.\]
\end{definition}
 
 \begin{remark}\label{r:surj}
 	(1) The map $\eval_B$ is the $p$-adic analog of the integration 
 	\[\int\colon \cP^\mathrm{mot}_\bC(M)\to \bC\]
 	 from the classical complex case.
	
	(2) By its very construction, the image of $\eval_B$ is  the algebra $\cP_p(\<M\>)$ from \cref{D:field of p-adic periods of <M>} which shows that $\cP_p(\<M\>)$ is finitely generated. 
\end{remark}

\begin{proposition}\label{T:trdeg <= dim H_M}
Let  $\cP_p(\<M\>)$ be the algebra introduced in  \cref{D:field of p-adic periods of <M>}. Then 
	we have the inequality
	\begin{align}\label{Eq:trdeg<=dimH_M}
		 \trdeg_\bQb(\cP_p(\<M\>))\le \dim(\cH_M).
	\end{align}
\end{proposition}
\begin{proof}
	As explained in  \cref{r:surj}, $\cO(\cH_M)$ surjects onto $\cP_p(\<M\>)$, which immediately gives the statement.	
	
	Alternatively, notice that the fraction field of $\cP_p(\<M\>)$ is the smallest subfield~ $L$ of $\bQb_p$ for which the map  ensuing from Berthelot's comparison isomorphism, 
	\[ \fcB(N)\colon Z_p(N)\otimes_\bQb L \to R_\dR(N)\otimes_\bQb L,\]
	is defined for all $N\in\<M\>$.
	In other words, it is the field of definition of the point $\fcB\in \cH_M(\bQb_p)$ and therefore it is also the residue field of the point $\fcB$.
	We conclude using the fact that the residue field of any point of $\cH_M$ has transcendence degree bounded above by the dimension of its Zariski closure in $\cH_M$.
\end{proof}
The above constructions give a tannakian framework for the study of the algebra~ $\cP_p(\<M\>)$. Keeping \cref{tensor periods} in mind our next task is to establish such a framework for  $\cP_p(M)$, the algebra of Andr\'e's $p$-adic periods of $M$, introduced in
 \cref{D:algebra of periods of M}. To do so we will start by giving a motivic interpretation of the numbers $[M,v,\xi]$, introduced in \cref{period triples}, as matrix coefficients.

\begin{definition}\label{motivic periods}
Fix a motive $N\in \<M\>$, a vector $v \in Z_p(\spe(N))$ and a covector $\xi \in R_\dR(N)^{\vee}$. To such a triple we associate a function  $[N,v,\xi]^{\mathrm{mot}} \in\cO(\cH_M) $, called the motivic $p$-adic period associated with the triple, defined on $A$-points, where $A$ is a $\obQ$-algebra, as
\[ \cH_M(A):=\Emb^\otimes(Z_p\circ \spe|_{\<M\>}, R_\dR|_{\<M\> })(A)  \longrightarrow A  \hspace{1cm} c \mapsto \xi(c(v)). \]

The algebra $\cAp(N)$ of motivic $p$-adic periods of $N$  is defined as the subalgebra 
	\[ \cAp(N) \subset \cO(\cH_M) \]
	 generated by the functions $[N,v,\xi]^{\mathrm{mot}}$, for all possible choices of $v$ and $\xi$.
\end{definition}
\begin{remark}\label{surj second}
 Consider the map of algebras
\[ \eval_{B}\colon \cAp(M)\to \obQ_p\]
given by evaluating functions at the point $\fcB\in\cH_M(\bQb_p)$. 
Then one has \[\eval_{B} ([N,v,\xi]^{\mathrm{mot}}) = [N,v,\xi]\] by its very construction; see \cref{period triples}.
In particular we have \[\eval_{B} (\cAp(N))  = \cP_p(N),\] see 
 \cref{D:algebra of periods of M}.

Notice also that the algebra $\cAp(N)$ is finitely generated. Indeed it is  generated by the functions $[N,v,\xi]^{\mathrm{mot}}$, with $v$ and $\xi$ varying respectively in two fixed basis of the vector spaces $ Z_p(\spe(N))$ and   $ R_\dR(N)^{\vee}$.
\end{remark}

\begin{proposition}\label{P:surjective HM to HN}
	Let $N$ be a motive in $\<M\>$. Consider the natural map \[\cH_M \longrightarrow \cH_N\]
	 induced by restricting $\Emb^\otimes(Z_p\circ \spe|_{\<M\>}, R_\dR|_{\<M\> })$ from the category $\<M\>$ to the category $\<N\>$. Then this map is surjective  and in particular the induced map on functions
	 \[ \cO(\cH_N) \hookrightarrow  \cO(\cH_M) \] is an injection of algebras.
\end{proposition}
\begin{proof}
The inclusion $\<N\>\subset \<M\>$ corresponds to a surjective map of de Rham tannakian groups $G_\dR(M)\to G_\dR(N)$ and
the  natural map $\cH_M \longrightarrow \cH_N$ is equivariant with respect to it. By \cref{T:cH_M is repr} these spaces are homogenous under these actions, hence we have the surjectivity.
\end{proof}
\begin{remark} 
 The function
$[N,v,\xi]^{\mathrm{mot}}$ introduced in \cref{motivic periods} makes sense both in $\cO(\cH_N) $ and in $\cO(\cH_M) $. By construction the natural map $ \cO(\cH_N) \hookrightarrow  \cO(\cH_M) $   sends $[N,v,\xi]^{\mathrm{mot}}\in \cO(\cH_N)$ to $[N,v,\xi]^{\mathrm{mot}}\in  \cO(\cH_M)$. For this reason we did not put any symbol $M$ in the definition of $[N,v,\xi]^{\mathrm{mot}}\in  \cO(\cH_M)$. Notice also that the injectivity of the natural map $ \cO(\cH_N) \hookrightarrow  \cO(\cH_M) $ implies that the relations between the motivic $p$-adic periods do not vary if we enlarge the category.
\end{remark}
\begin{proposition}\label{C:trdegM<=dimSpecA_p}
	We have an inequality
\begin{align}\label{Eq:trdegM<=dimSpec}
		 \trdeg_\bQb(\cP_p(M))\le \trdeg_\bQb(\cAp(M)) = \dim\Spec(\cAp(M)).
\end{align}
\end{proposition}
\begin{proof}
The argument is analogous to the one in the proof of \cref{T:trdeg <= dim H_M}.
\end{proof}

 We end the section with some examples continuing those discussed in \cref{SS:periods}.

\begin{example}\label{Eg:H_M for E=CM new} (Continuing \cref{Eg:H_M for E=CM}.)
	Let $E/\bQb$ be a CM elliptic curve  with  supersingular reduction $E_p$ modulo $p$.
	Denote by $L$ its CM field. 
	Consider the motive $M:=\cEnd(E)$.
	
	We have $G_\dR(E) \simeq \bG_m^2$, $G_\dR(M)=\bG_m$ and $G_\crys(M)=1$. 
		The space of periods is  $\cH_M \simeq \bG_m$.
		The matrix coefficients of $M$ generate $\cO(\cH_M)$ because $M$ corresponds to the natural representation of $G_\dR(M)\simeq \bG_m$. 
		As such we get that $\cAp(M)=\cO(\cH_M)\simeq\bQb[t,t\inv]$ and hence that $\cP_p(M)=\cP_p(\<M\>)$, i.e., the periods of $M$ generate the periods of $\<M\>$.
\end{example}

\begin{example}\label{Eg:H_M for E=ordinary new} (Continuing \cref{Eg:H_M for E=ordinary}.)
	Let $E/\bQb$ be a non CM elliptic curve and consider the endomorphism motive $M:=\cEnd(E)$.
	We have $G_\dR(M)\simeq \PGL(2)$.
	Now consider two cases: 
	\begin{itemize}
		\item $E_p$ supersingular. We have  $G_\crys(M)=1$ hence $\cH_M \simeq \PGL(2)$.
		Moreover, we have $\cAp(M)=\cO(\cH_M)$ the reason being that matrix coefficients for a faithful representation of a reductive group generate the matrix coefficients of all representations.
		In particular, we obtain $\cP_p(M)=\cP_p(\<M\>)$.

		\item $E_p$ not supersingular. Here we have $G_\crys(M)\simeq \bG_m$ (we can see it as the maximal torus $T$ in $\PGL(2)$)
		hence we have $\cH_M\simeq \PGL(2)/\bG_m$.
		Here too, by computing the matrix coefficients map 
		\[ V^T\boxtimes V^*\to \cO(\PGL(2)/T) \]
		for $V$  the adjoint representation of $\PGL(2)$, one can check the equality $\cAp(M) = \cO(\cH_M)$.
		In particular, one has $\cP_p(M) = \cP_p(\<M\>)$, i.e., the periods of $M$ generate the periods of $\<M\>$.
	\end{itemize}
\end{example}

\begin{remark}
	We briefly interrupt the series of examples to highlight a particular point about matrix coefficients which might clarify the difference between the algebras $\cAp(M)$ and $\cAp(\<M\>)$.
	
	Let $G$ be a reductive group over $k$, an algebraically closed field of characteristic zero.
	For any representation $V$ of $G$ we have the matrix coefficients map
	\begin{align*}
		m_V\colon &V\boxtimes V^\vee \to \cO(G)\\
		& v\boxtimes \xi\mapsto (g\mapsto \xi(gv)).
	\end{align*}
	
	The Peter--Weyl theorem gives an isomorphism of $G\times G$-modules
	\begin{align}	\label{eq:Peter-Weyl iso}
	\bigoplus_{V\text{ irr}} V\boxtimes V^\vee \simeq  \cO(G)
	\end{align}
	where $V$ runs over all irreducible representations of $G$.
	Moreover, for any faithful irreducible representation $V$, the summand $V\boxtimes V^\vee$ generates $\cO(G)$ as an algebra (see \cite[\S3.5 Theorem]{Wat}).
	
Now let $H\le G$ be a closed subgroup. The map \eqref{eq:Peter-Weyl iso} restricts to an isomorphism of $G$-modules
\[ \bigoplus_{V\text{ irr}} V^H\boxtimes V^\vee \simeq  \cO(G/H).\]

However, now a summand $V^H\boxtimes V^\vee$ might not generate the algebra $\cO(G/H)$. 
An extreme example is for irreducible representations $V$ of $G$ such that $V^H=\{0\}$. This is the case for $\bG_m\le \SL(2)$ and any even dimensional irreducible representation. A less extreme example is mentioned below for $\bG_m\le \PGL(2)$ and $\Sym^4k^2$.
We do not know for which irreducible representations $V$ the matrix coefficients of $V^H\boxtimes V^\vee$ generate the algebra $\cO(G/H)$. 

In the motivic setting, the algebra $\cAp(\<M\>)$ is of the form $\cO(G/H)$ whereas $\cAp(M)$ corresponds to the subalgebra generated by $V^H\boxtimes V^\vee$ where $V=R_\dR(M)$ viewed as a $G=G_\dR(M)$-module (see \cref{motivic periods}).

This phenomenon does not happen in the classical case of complex periods because $\cP_\bC^{\mathrm{mot}}(\<M\>)$ is of the form $\cO(G)$.
\end{remark}

\begin{example}\label{eg:sym4 no Kunnet for periods}
	Let $E$ be a non CM elliptic curve with ordinary reduction modulo~ $p$. Consider the motive $N:=(\Sym^4\fh^1(E))(2)\in \<M\>$ where $M=\cEnd(E)$. 
	We recall that $G_\dR(M)=\PGL(2)$ and $G_\crys(M)=\bG_m$.
	It is also the case that $\<N\>=\<M\>$ and so $\cH_N=\cH_M \simeq \PGL(2)/T$.
	By a direct, cumbersome computation of the matrix coefficients map
	\[ V^T\boxtimes V^*\to \cO(\PGL(2)/T) \]
	for $V=\Sym^4(\bQb^2)$ we find that $\cAp(N)\subsetneq \cO(\cH_N) = \cAp(\<N\>)$. 

	 
	If the analog of the Grothendieck period conjecture holds, namely if the evaluation map $\eval_{B}\colon \cAp(\<N\>) \to \cP_p(\<N\>)$ is an isomorphism (see \cref{P:equiv of GPC}), then we would have $\cP_p(N)\subsetneq \cP_p(\<N\>)$.
\end{example}

\section{The ramification condition}\label{SS:ramification}
We start this section by recalling an example (see \cref{Eg:Andre CM ell curve}), due to Andr\'e, which in particular shows that   the Berthelot point
 $\fcB\in \cH_M(\bQb_p)$ from \cref{BO point} is not always a generic point of the space $\cH_M$ introduced in \cref{D:of H_M}.
 
 Following again Andr\'e, who studied similar questions for the $H^1$ of abelian varieties, one can hope that this kind of examples is sporadic; see \cref{construction andre}.
  He formulated a ramification condition for  the $H^1$ which allows one to avoid this pathological example.

 The goal of this section is to find a generalization of this ramification condition which makes sense for any motive (\cref{D:p ramifies in <M>}).
 We show that such a condition is generically satisfied (\cref{C:p does not ramify in <M>}).
 
\noindent In order to formulate this ramification condition we need to work with coefficients in $\bQ$ (contrary to the previous sections, where the field of coefficients was $\bQb$).

\begin{example}[{\cite[I.4.6.4]{Andre-p}}]\label{Eg:Andre CM ell curve}
Let $E$ be a CM elliptic curve over $\bQb$ with  $\End(E)\otimes_\bZ\bQ = \bQ[\sqrt{-3}]$.  Let $p$ be a place above $3$. 
	Then the reduction modulo $p$ of $E$ is   supersingular.
	Consider the motive $M:=\cEnd(\fh^1(E)) = \fh^1(E)\otimes \fh^1(E)^\vee$ and its reduction modulo $p$.
 	In \textit{loc. cit.}, André shows  that all the $p$-adic periods of $M$ are algebraic numbers (see the ramified case of \cref{Eg:H_M for E=CM} for more details). On the other hand, the homogeneous space $\cH_M$  has dimension one (\cref{Eg:H_M for E=CM new}). This means that the transcendence degree of the algebra generated by the periods is strictly smaller than the dimension of $\cH_M$  (equivalently, the Berthelot point  does not live on the generic point of $\cH_M$).
	
	In the above example the prime number $ 3$ ramifies in the number field ${\bQ[\sqrt{-3}] = \End(E)\otimes_\bZ\bQ}$.  
	Andr\'e conjectures that this ramification condition should be the only situation where all the periods of the motive $M:=\cEnd(\fh^1(E))  $ are algebraic (see \cref{construction andre} for more details).
\end{example}

\begin{remark}(Relation to Gamma functions.)
	In the classical complex setting, periods of abelian CM motives are related to products of special values of the Gamma function (the Gross--Deligne conjecture \cite[p.209]{Gross}).
	There are explicit relations between such special values  \cite[Appendix]{Deligne-valL-periodes}
	which are conjectured to be the only ones.
	Such a prediction is coherent with the Grothendieck period conjecture.
	
	In the $p$-adic context the situation is more subtle as \cref{Eg:Andre CM ell curve}  shows. There seems to be a discrepancy between the ramified and the unramified case, namely  in the ramified situation the algebricity of the $p$-adic periods is not fully explained by motives.
	The reason seems to boil down to the fact that the $p$-adic periods have different expressions in terms   of special values of the $p$-adic gamma function in the two cases (see  \cref{Eg:H_M for E=CM} for the two expressions).

\end{remark}

\begin{remark}(André's Betti lattices and ramification.)\label{construction andre} In \cite{Andre1} (and later \cite[I.5.3]{Andre-p}),
Andr\'e investigated the possibility of constructing a second  $\bQb$\nobreakdash-structure to crystalline cohomology -- especially for the $H_\crys^1(A)$, with $A$ an abelian variety -- to be compared to the  $\bQb$-structure given by the de Rham cohomology $H^1_\dR(\tilde{A})$ of a given lifting $\tilde{A}$. 
Such a canonical structure should be called "Betti", and denoted $H^1_B$. 
It should come with two properties: first the endomorphisms of $A$ should act on $H^1_B$ and second any algebraic classes on any power of $A$ should be rational with respect to this Betti structure.
He proposed a construction in two cases:
\begin{enumerate}
\item If $A$ is ordinary \cite[\S 5]{Andre1}. In this case one can take the de Rham cohomology of the Serre--Tate lifting.
 (This is interesting if one started with a non CM lift $\tilde{A}$  -- in particular not the Serre--Tate lifting itself).
\item  If $A$ is supersingular. In this case he shows  that there exists a unique Betti structure $H^1_B(A)$ with the property that the endomorphisms of $A$ preserve it and that $\wedge^2H^1_B(A)(1)\subset H_\crys^2(A)(1)$ is the space of algebraic classes 
\cite[Proposition I.5.3.3]{Andre-p}.
\end{enumerate}
The proposed $\bQb$-structure $H^1_B$ is related to the one we are studying 
in an indirect way; see \cref{Eg:H_M for E=ordinary} for the supersingular reduction case and   \cref{Eg:two ell curves} for the ordinary reduction case.

In \cite[Remark in \S 5]{Andre1} Andr\'e speculates (at least in the case of elliptic curves) that the two $\bQb$-structures of $H_\crys^1(A)$, Betti and de Rham, should not coincide.
In \textit{loc. cit.} he pointed out that  if $A$ is supersingular and $\tilde{A}$ is  CM then these two $\bQb$-structures are related through  values of the $p$-adic Gamma function (by \cite{Ogus}, based on computations of Coleman \cite{Colem}). 

On the other hand, as he   afterwards noticed, those special values of the Gamma function might turn out to be algebraic by the Gross--Koblitz formula \cite[I.4.6.4]{Andre-p}.
He conjectured that this phenomenon should be pathological and formulated the following conjecture \cite[Conjecture I.5.3.10]{Andre-p}
 (the section of the conjecture is written  for elliptic curves "for simplicity", as the author says  in \textit{loc. cit.} ).
  If  $A$ is supersingular  and $\tilde{A}$ is not CM (or if the prime number $p$ does not ramify in the CM field) then  the two $\bQb$-structures of $H_\crys^1(A)$, Betti and de Rham, should not coincide.
\end{remark}

\begin{definition}\label{D:CM motive}
	A motive $N\in\Mot(\bQb)_\bQ$ is called CM if $\End(N)$ is a number field of degree equal to the rank of $N$.
\end{definition}
The name CM is justified by the remark below. Those motives are needed in order to generalize  Andr\'e's  ramification condition from abelian varieties to general motives; see the definition which follows.

\begin{remark}(Not needed in the sequel.) The standard conjectures predict that if the motive $N$ is a CM motive with the above definition, then $\End(N)$ has a positive involution, hence it should indeed be a CM number field.
\end{remark}
\begin{definition}\label{D:p ramifies in <M>}
	Let $M\in\Mot(\bQb)_\bQ$ be a motive, let $\<M\>$ be the tannakian category it generates  and let $p$ be a prime number.
	We say that $p$ ramifies in $\<M\>$ if there exists a CM motive $N\in\<M\>$ such that the prime  $p$ ramifies in the number field  $\End(N)$. 
	If $p$ does not ramify in $\<M\>$ we say that the motive $M$ is unramified above~$p$.
\end{definition}

\begin{remark}
	
	In the \cref{Eg:H_M for E=ordinary,Eg:two ell curves} the prime $p$ is not ramified in $\<M\>$. The reason is that the tannakian groups are $\PGL(2)$ and $\GL(2)$ and so the only CM objects in $\<M\>$ are the Tate twists (see for example the proof of the next proposition). For Tate twists the ramification condition is empty as $\End(\one(n)) = \bQ$.
\end{remark}

\begin{proposition}\label{P:field L for all CM obj}
	Let $M\in\Mot(\bQb)_\bQ$ be a motive. Then there exists a number field $L_M$ such that for all CM objects $N\in\<M\>$ we have $\End(N)\hookrightarrow L_M$.
\end{proposition}
\begin{proof}
	We consider the Betti realization $R_B$ (\cref{def:realization}) restricted to $\<M\>$ and denote the corresponding tannakian group by $G_B(M)$.
	It is an algebraic group over~ $\bQ$ and the analog of \cref{T:tann for <M>} says that $R_B$ induces an equivalence of tensor categories $R_B\colon \<M\>\to \Rep_\bQ(G_B(M))$.
	Therefore we can place ourselves in the representation category $\Rep_\bQ(G_B(M))$ and, to ease the notation, we will write $G$ instead of $G_B(M)$.
	
	Let $V\in\Rep_\bQ(G)$ be a CM object, i.e., $K:=\End_G(V)$ is a number field of degree equal to $\dim(V)$.
	This turns $V$ into a $K$-vector space of dimension 1 and moreover the action of $G$ on $V$ commutes with $K$.
	Since $\End_K(V)$ is abelian, we deduce that the action of $G$ on $V$ factorizes through its abelianization $G/[G,G]$.
	We can therefore assume from now on that $G$ is an abelian algebraic group.
	
	One can write $G$ as a direct product $U\times S$ where $U$ is unipotent (abelian) and $S$ is a diagonalizable group, possibly disconnected if $G$ is. (For pure motives, $G$ is reductive and hence $U=\{1\}$. We do not assume it here so it applies also in the context of mixed motives in \cref{S:MTM}.)
	
	Since $U$ is abelian it is isomorphic to $\bG_a^r$ for some $r\ge 0$.
	Let $L$ be a finite extension of $\bQ$ such that $S$ splits over it, i.e.,  $S_L\simeq \bG_m^s\times \Gamma$ for some $s\in\bN$ and some finite group $\Gamma$.
	Extend the group $G$ to $L$ and notice that all its irreducible representations on $L$-vector spaces are of dimension 1.
	
	Put $V_L:=V\otimes_\bQ L$ and note that $V_L$ is semisimple as a representation of $G_L$ (see \cite[Theorem 7.5]{CurRei-I}).
	All simple representations of $G_L$ are of dimension one over $L$, therefore $V_L$ splits as a direct sum of representations of dimension one.
	Using the equalities $\End_{G_L}(V_L)=\End_G(V)\otimes_\bQ L= K\otimes_\bQ L$ we deduce that $\End_{G_L}(V_L)$ is commutative, hence isomorphic to $L^m$ for some $m\in\bN$. 
	It follows that $\End_G(V)$ embeds in the field  $L$ and this concludes the proof.	
\end{proof}
\begin{corollary}\label{C:p does not ramify in <M>}
	For $M\in\Mot(\bQb)_\bQ$ there exists a field $L_M$ such that if $p$ does not ramify in $L_M$, then it does not ramify in $\<M\>$. In particular there are only finitely many prime numbers that ramify in $\<M\>$.
\end{corollary}

\section{The $p$-adic analog of the Grothendieck period conjecture}\label{SS:pGPC} 
We now have all the ingredients to formulate the $p$-adic analog of the Grothen\-dieck period conjecture. 
In analogy to the classical case we will give the strong and the weak formulation.
On the other hand, contrary to the classical case, the strong version has itself two non equivalent versions, each of those regarding a different algebra of Andr\'e's $p$-adic periods (this difference with the classical case is discussed in \cref{tensor periods}).

Throughout the section, we work with a motive with good reduction  with rational coefficients $M\in \MotgrQ$ (see \cref{S:conventions}). 
Rational coefficients are crucial here because we want to ask $M$ to be unramified above $p$ (\cref{D:p ramifies in <M>}).

By extending the scalars from $\bQ$ to $\bQb$ (in the sense of \cref{S:conventions}\eqref{extending scalars}) we deduce a motive, still denoted by $M$, in the category $\Motgr=\Mot(\bQb)^\textrm{gr}_\bQb$.
To such a motive we can attach  two algebras  of  Andr\'e's $p$-adic periods $\cP_p(M) \subset \cP_p(\<M\>)$ (see \cref{D:matrix of periods of M}  and \cref{D:field of p-adic periods of <M>}),
the space of motivic $p$-adic periods $\cH_M$ (\cref{D:of H_M}) and 
the algebras of motivic $p$-adic periods  $\cAp(M) \subset \cO(\cH_M) $  of $M$
	 (\cref{motivic periods}).
 
\begin{conjecture}\label{Conj:p adic analog GPC}
	Suppose that  $M\in   \MotgrQ$   is unramified above $p$.
	Then we conjecture that  $\Spec(\cAp(M))$ is an irreducible variety and that the inequality
\begin{align*}
	 \trdeg_\bQb(\cP_p(M))&\le \dim \Spec(\cAp(M)) 
\end{align*}
from    \cref{C:trdegM<=dimSpecA_p}
is in fact an equality.
\end{conjecture}
\begin{remark}
	We do not know if the ramification condition is optimal but it makes sense to impose it in order to avoid the situation from \cref{Eg:Andre CM ell curve}.
\end{remark}

\begin{conjecture}\label{Conj:p adic analog GPC tensor}
	Suppose that  $M\in   \MotgrQ$   is unramified above $p$.
	Then we conjecture that $\cH_M$ is an irreducible variety and that the inequality
\begin{align*}
	 \trdeg_\bQb(\cP_p(\<M\>))&\le \dim(\cH_M)
\end{align*}
from    \cref{T:trdeg <= dim H_M}  
is in fact an equality.
\end{conjecture}

Let us give some equivalent formulations of \cref{Conj:p adic analog GPC} and \cref{Conj:p adic analog GPC tensor} and study their relation.

Recall from \cref{BO point} that  we have the  Berthelot point
$\fcB\in \cH_M(\bQb_p)$ and  the evaluation map
$ \eval_{B}\colon \cO(\cH_M)\to \bQb_p $
 which, by \cref{surj second}, verifies 
 \[\eval_{B}(\cAp(M)) = \cP_p(M).\]
\begin{proposition}\label{P:equiv of GPC}
	Let $M\in\Mot(\bQb)_\bQb$. 
	\begin{enumerate}
			\item The following are equivalent:
		\begin{enumerate}
		\item \cref{Conj:p adic analog GPC tensor} holds for $M$,  
		\item \label{Conj:injective} $\eval_{B}\colon \cO(\cH_M)\to \bQb_p$ is injective,
		\item $\cH_M$ is irreducible and the  point
 $\fcB\in \cH_M(\bQb_p)$ lives above the generic point.
		\end{enumerate}
	\item The following are equivalent:
	\begin{enumerate}
		\item \cref{Conj:p adic analog GPC} holds for $M$,
		\item $\eval_{B}\colon \cAp(M)\to \bQb_p$ is injective.
	\end{enumerate}
 	\end{enumerate}
\end{proposition}
\begin{proof}
	We only prove the equivalence between $(a)$ and $(b)$ in the first part. 
	The proof of the second is analogous. 
	The equivalence between $(a)$ and $(c)$ in the first part is standard.
	
	Put $I:=\ker(\eval_{B}\colon \cO(\cH_M)\to \bQb_p)$  
	and consider the isomorphism of algebras
	 \begin{align*}\label{Eq:O/I=P}
	 	\cO(\cH_M)/I \simeq \cP_p(\<M\>).
	 \end{align*}
	 As $\cP_p(\<M\>)$ is an integral $\bQb$-algebra,
 	  $I$ must be a prime ideal.
 	
	If $I=(0)$, then $\cO(\cH_M)$ is integral.
	Since $\dim(\Spec(A)) = \trdeg_\bQb(A)$ for any integral $\bQb$-algebra, we deduce
	the equality $\trdeg_\bQb(\cP_p(\<M\>))=\dim\cH_M$. 
	
	Conversely, since $I$ is a prime ideal, the assumption (a) implies that $I$ must be of height $0$ in an integral algebra, so equal to $(0)$.
\end{proof}

\begin{corollary}
\cref{Conj:p adic analog GPC tensor} holds for $M$ if and only if \cref{Conj:p adic analog GPC} holds for all $N \in \< M \>.$
\end{corollary}
\begin{proof}
We use the description \eqref{Conj:injective} of the conjectures coming from \cref{P:equiv of GPC}.
	It is clear that the injectivity of 
	\[ \eval_{B}\colon \cO(\cH_M)\to\bQb_p \]
	 implies the same for its restriction to $\cAp(M)$.
	 
	Conversely, recall that $\cH_M\simeq G/H$ where $G=G_\dR(M)$ and $H=G_\crys(M)$ and that $G$ is a reductive group.
	The algebra $\cO(G)$ of regular functions on $G$, viewed as a $G\times G$-representation, decomposes as
	\[ \cO(G) = \bigoplus_{V\text{ irr}} V\boxtimes V^\vee\]
	which implies the decomposition of $\cO(\cH_M)$ as a $G$-representation:
	\[ \cO(\cH_M)\simeq \bigoplus_{V\text{ irr}} V^H\boxtimes V^\vee.\]
	The map $V^H\boxtimes V^\vee\to \cH(\cH_M)$ is precisely the matrix coefficients map that we considered in \cref{motivic periods}.

	In our setting this implies that $\cO(\cH_M)$ is generated by $\cAp(N)$ for $N\in\<M\>$ (actually finitely many $N$ suffice since $\cO(\cH_M)$ is a finitely generated algebra). 
	We can deduce that the injectivity of $\eval_{B}$ when restricted to $\cAp(N)$ for sufficiently many $N\in\<M\>$ implies the injectivity of $\eval_{B}$ on $\cO(\cH_M)$.
\end{proof}

We now turn to the weak version of the $p$-adic  Grothendieck period conjecture.
 Recall the inclusion  $Z_p(\spe(M)) \subset R_\crys(\spe(M))$  induced by  $\cl_p(\spe(M))$ (\cref{def:realization}) and the inclusion  $R_\dR(M) \subset R_\crys(\spe(M))$ induced by the Berthelot comparison isomorphism \eqref{Eq:the Berthelot compar}.

\begin{conjecture}\label{Conj:p adic analog GPC weak}
Suppose that  $M\in   \MotgrQ$   is unramified above $p$. 
	Then we conjecture that the inclusion of $\bQb$-vector spaces
	\[  Z_0(M)_\bQb \subset    R_\dR(M) \cap Z_p(\spe(M))_\bQb  \]
	 is in fact an equality.
\end{conjecture}
In    \cref{Conj:liftp} we were interested in the above equality with coefficients in $\bQ$. Here we give  a simple argument which allows one to go from $\bQb$ to $\bQ$.
\begin{proposition}\label{Q vs Qbar}
Fix a motive $M\in \MotgrQ$ that is unramified over $p$ and that satisfies \cref{Conj:p adic analog GPC weak}.

Then the inclusion of $\bQ$-vector spaces
	\[  Z_0(M)_\bQ \subset    R_\dR(M) \cap Z_p(\spe(M))_\bQ  \]
is an equality.
 \end{proposition}
 \begin{proof}
 Consider the commutative diagram 
 \[\begin{matrix}
Z_0(M)_\bQ \otimes_\bQ \bQb &  \subset & (R_\dR(M) \cap Z_p(\spe(M))_\bQ) \otimes_\bQ \bQb  \\
\downarrow &  & \downarrow \\
Z_0(M)_\bQb &  = & R_\dR(M) \cap Z_p(\spe(M))_\bQb
\end{matrix}\]
 The first vertical arrow is an isomorphism because   extending coefficients  commutes with numerical equivalence \cite[Proposition 3.2.7.1]{Andmot}  or alternatively it commutes with homological equivalence for classical cohomologies in characteristic zero (one reduces to singular cohomology through comparison theorems).
 
 The second vertical arrow is injective because it is even injective when one extends to $\bQb_p$-coefficients.
 This comes from the Berthelot comparison isomorphism \eqref{Eq:the Berthelot compar} and again the fact that extending coefficients   commutes with numerical equivalence.

 Hence, by passing through the first vertical arrow we deduce that the diagonal composition is an isomorphism, whereas by passing through the second vertical map we have that the diagonal is a composition of injective maps. We deduce that all arrows are isomorphisms and therefore the inclusion is an equality.
 \end{proof}
\begin{remark}
 \cref{Conj:p adic analog GPC weak}  with \cref{Q vs Qbar} corrects \cref{Conj:liftp}. 
 Notice that when we discussed  \cref{Conj:liftp}  the ramification condition was not yet introduced.
	The relation of this conjecture with other classical conjectures is sketched in \cref{rem:bcconj}.
\end{remark}

\begin{proposition}\label{strong implies weak}
The validity of \cref{Conj:p adic analog GPC tensor} for $M$ implies the validity of 
\cref{Conj:p adic analog GPC weak} for all $N \in \<M \>.$ 
\end{proposition}
\begin{proof}
Let us make a tannakian translation of the problem, as in the proof of \cref{T:cH_M is repr precise}. 

Let $G$ be an algebraic group over an algebraically closed field $k$ and let $H\le G$ be a closed reductive subgroup of $G$.
We consider the diagram of categories
\[\begin{tikzcd}
	\Rep_k(G) \arrow[r,"\Forg"] \ar[d,"\spe", swap]& \Vect\\
	\Rep_k(H) \ar[ru, swap,"(-)^H"]
\end{tikzcd}  \]
where $\spe$ is the restriction functor.
Consider  the functor 
\[\cH:=\Nat^\otimes((-)^H\circ\spe, \Forg)\] and a point $\alpha\in\cH(K)$,   where $  K$ is an algebraically closed field extension of $k$.
(Recall that $\cH$ is representable by the variety $G/H=\overline{G/H}$; see \cref{T:Haines,T:cH_M is repr precise}.)

In the case we are interested in we want to apply this abstract setting to $G=G_\dR(M)$, $H= G_\crys(M)$, $k=\bQb$, $K=\bQb_p$ and $\alpha=\fcB$ the Berthelot point.

With this tannakian translation, the contents of \cref{Conj:p adic analog GPC tensor} and
\cref{Conj:p adic analog GPC weak} become, respectively:
\begin{enumerate}
	\item[(1)] $\alpha$ is generic in $\cH$\\
	\item[(2)] for all representations $V\in\Rep_k(G)$ we have that the natural inclusion
	\[ V^G\subset \alpha(V^H)\cap V\]
	is an equality (the intersection is considered inside $V\otimes_k K$).
\end{enumerate}
	We now prove that (1) implies (2). For that, suppose that there exists a representation $V\in\Rep(G)$ for which (2) fails, i.e., that there exists $v\in V^H$ such that $\alpha v\in V\setminus V^G$.
	
	Since $\alpha\in\cH(K)=(G/H)(K)$ we can lift it (non canonically) to an element  $\alpha\in G(K)$.
	Notice now that there exists a $g_0\in G(k)$ such that $g_0v\neq \alpha v$ otherwise we would have $v = \alpha v \in V^G$. Pick moreover a linear function $\lambda\in V^{\vee}$ such that $\lambda(g_0v-\alpha v)\neq 0$.
	We define the function
	\[ f\colon G\to k \qquad f(g):=\lambda(gv-\alpha v). \]
	By construction, $f$ is regular and also right $H$-invariant because $v\in V^H$.
	Moreover it is not zero  as the value $f(g_0)$ is not zero.
	Hence $f$ descends to a non zero regular function $f \in  \cO(G/H)=\cO(\cH)$.
	On the other hand, again by construction, we have $f(\alpha)=0$ which implies that  $\alpha$ is not generic in $\cH$.
\end{proof}

\begin{remark}
	As for classical periods, there is no converse of the above proposition.
	Namely, in the abstract tannakian setting of the proof,  (2) does not imply (1) in general.
	
	For an example, consider the trivial subgroup $H= 1 < \GL(2)(k)$.
	Notice that $\cH = G$. 
	Consider $a,b\in K$ algebraically independent over $k$ and consider the matrix
	\[ \alpha:=\begin{pmatrix}
		a & 1\\
		1 & b
	\end{pmatrix}\in G(K). \]
	We claim that for all representations $V\in \Rep_k(G)$ we have 
	\[ \alpha(V)\cap V=\{0\}  \text{ inside } V\otimes_k K.\]
	Notice first that it is enough to check it on irreducible representations. 
	For $\GL(2)$ they are all of the form $\det^s\otimes \Sym^r k^2$.
	A non enlightening computation shows that for such representations the claim is true.
	 Thus (2) holds in this example.

	However, $\alpha$ is not generic in $G$ as its Zariski closure has dimension two, hence (1) does not hold.
\end{remark}

\begin{remark}
	\cref{strong implies weak} has an analog for classical periods; see for example  \cite[Proposition 7.5.2.2]{Andmot}. (Notice that \textit{loc. cit.} applies under our hypothesis; see \cref{S:conventions}\eqref{hom=num}.)
	The standard way of showing it is to consider the de Rham--Betti realization, which lands in the category of triples $(V,W,c)$ where $V$ and $W$ are $\bQb$-vector spaces and $c$ is a comparison isomorphism $V\otimes \bC \simeq W\otimes \bC.$
	A key point is that this category is tannakian.
	
	In the $p$-adic setting one would like to consider the category of triples $(V,W,i)$ where $V$ and $W$ are $\bQb$-vector spaces and $i$ is an injection $V\otimes \bQb_p \hookrightarrow W\otimes \bQb_p.$
	However, this new category fails to be tannakian (it is not even abelian because cokernels do not exist), hence the classical approach does not apply to \cref{strong implies weak}.
\end{remark}

\section{Mixed case}\label{S:MTM}

In this section we sketch how to adapt the constructions of the paper in the setting of mixed motives as developped by Voevodsky et al. \cite{Voe,Ayoub,CD}.

There are two subtleties of different level.
The first is that for a non proper scheme over $\cO_{\fp}$ the Berthelot comparison theorem \eqref{Eq:the Berthelot compar} does not hold in general (e.g., the
 special fiber by $X_p$ can be empty).
 This problem is solved   considering the subcategory of motives verifying this property, which contains for instance the extensions of motives of smooth proper varieties.
 
The second issue is more serious; we need to work with a subcategory of Voevodsky's motives $\DM(\cO_{\fp})$ which is  tannakian and for which the realization functors are fiber functors. 
But in the mixed context, the only case where this is known is for mixed Tate motives over the ring of integers of a number field.
 
On the other hand, the case of mixed motives seems to be a setting where plenty of interesting periods should appear. 
This comes from the fact that over $\bQb$ there are lots of non trivial extensions of motives, whereas over $\bFb_p$ all extensions conjecturally split,  giving rise to non liftable maps and, in particular, to non liftable  algebraic classes (see \cref{weight}).


\begin{remark}
There exists a triangulated tensor subcategory $\DM^{B}(\cO_{\fp})$ of $\DM(\cO_{\fp})$ containing the motives of smooth proper schemes and for which the Berthelot comparison theorem \eqref{Eq:the Berthelot compar} holds. Indeed by \cite[Remark 4.5]{Vezzani2} there is a natural monoidal transformation (denoted by $\alpha$ in \textit{loc. cit.})   between the crystalline and de Rham realization functors.  In particular the subcategory of $\DM(\cO_{\fp})$ on which the transformation is invertible is indeed a tensor triangulated subcategory. The fact that it contains motives of smooth proper schemes is again in \cite[Remark 4.5]{Vezzani2}.
\end{remark}

\begin{notation}\label{N:assmp mixed motives}
Consider from now on $\mathcal{D}\subset \DM^{B}(\cO_{\fp})$ a sub-triangulated category endowed with a $t$-structure for which the realization functors are $t$-exact.
Write $\mathcal{A}\subset \mathcal{D}$ for the heart and suppose that the realization functors are fiber functors making $\mathcal{A}$   a tannakian category.
Suppose moreover that the restriction of  $\mathcal{A}$ to the special fiber forms a semisimple abelian category.
\end{notation}
\begin{remark}
An example of $\mathcal{A}$ as in \cref{N:assmp mixed motives} is the category of mixed Tate motives \cite{Levine}.
It is conjectured that the whole $\DM$ should be the derived category of an abelian category and that the realization functors should be exact \cite[21.1]{Andmot}.
It is also conjectured that such an abelian category should be semisimple over finite fields.
\end{remark}

\begin{theorem}\label{thmmix}
Let $ \mathcal{A}$ be a category of mixed motives as in \cref{N:assmp mixed motives} and let $M$ be a motive in  $ \mathcal{A}$.
Then all the constructions we made in the previous sections can be adapted to this context, in particular:
\begin{enumerate}
\item The functors and natural transformations $R_\dR,  R_\crys, Z_0,Z_p, \cl_0$ and $\cl_p$  (see \cref{def:realization}).
\item The algebras of Andr\'e's $p$-adic periods $\cP_p(M) \subset \cP_p(\<M\>)$ (see   \cref{D:algebra of periods of M} and \cref{D:field of p-adic periods of <M>}).
\item The tannakian groups  $G_\crys(M) $ and $ G_\dR(M)$  (constructed in 
\cref{T:tann for <M>}).
\item The homogenous space $\cH_M$ (see \cref{D:of H_M}). 
\end{enumerate}
Moreover, the homogenous space  $\mathcal{H}_M$ is representable and the following inequality holds
\[
		 \trdeg_\bQb(\cP_p(\<M\>))\le \dim(\cH_M).
	\]
(see \cref{T:trdeg <= dim group}).

Finally the ramification condition makes sense in this setting (\cref{D:p ramifies in <M>}) and the $p$-adic analogues of (the different versions of) the Grothendieck period conjecture can be formulated (\cref{SS:pGPC}).
\end{theorem}

\begin{remark}
The definitions and the proofs  in \cref{thmmix} are identical to the pure case.
Let us just mention that in the proof of \cref{T:cH_M is repr precise} we used that $G_\crys(M) $ is reductive.
Notice that it is again the case here (even though $G_\dR(M) $ will not be reductive in general) because the category of mixed motives over $\bFb_p$ is semisimple (under our assumptions from \cref{N:assmp mixed motives}).
\end{remark}

\begin{example}\label{ex log}
(Based on \cite[\S 2.9-2.10]{Delignelog} and \cite[Examples 2 and 3]{Yamashita})
Let $a $ be a rational number and $K_a$ be the Kummer motive $K_a=H^1(\mathbb{G}_m,\{1,a\})$.
When $a\neq 0,1,-1$,  it fits in an exact sequence
\[0 \rightarrow \one  \rightarrow K_a \rightarrow \one(-1) \rightarrow 0,\]
which is not split as these extensions are classified by $a\in \bQ^{\times} \otimes_\bZ \bQ$.

This motive has an algebraic class whereas its  dual $K_a^\vee$ does not have algebraic classes (otherwise the exact sequence would split).
For this reason one  expects the $p$-adic periods of $K_a^\vee$ to be interesting numbers. 
In what follows we compute them and we make a comparison with the classical ones\footnote{Here we will write the Betti cohomology with respect to the de Rham cohomology (and not the inverse). 
This seems the best analog to \cref{D:matrix of periods of M}.
Notice that this is the convention of Deligne  and it is the inverse of the one of Andr\'e \cite[Remark 23.1.1.1]{Andmot}.}.

The matrix of complex periods  is
\[ \mathcal{P}_\bC(K^\vee_a):=\begin{pmatrix}
		2\pi i & \log (a) \\
		0 & 1
	\end{pmatrix}, \]
	where we  chose bases for Betti and de Rham realizations adapted to the weight and the Hodge filtrations.
	
Let us now compute the $p$-adic periods. 
(To have a prime of good reduction, $a$ must be invertible modulo $p$.)
	First, the specialization $\spe(K^\vee_a)$ splits canonically
	\[\spe(K^\vee_a)=\one(1) \oplus \one\]
	hence the space of algebraic classes \[Z_p(\spe(K^\vee_a))=\Hom( \one , \spe(K^\vee_a)) = \Hom(\spe(K_a), \one)\] has dimension one and has a canonical generator given by
$(0,\id)$. This corresponds to the unique map $f:\spe(K_a) \rightarrow \one$ that is Frobenius equivariant and is the identity when restricted to $\one$.

	On the other hand, the action of the absolute Frobenius on crystalline cohomology can be computed.
	Its matrix (with respect to the same basis of the de Rham cohomology used above) is 
	\[ \varphi(K_a):=\begin{pmatrix}
		1 & \log (a^{1-p}) \\
		0 & p
	\end{pmatrix}. \]
	This means that the unique map $f$ written in the same basis is
	\[ f=\begin{pmatrix}
		1 & \frac{\log (a^{1-p})}{1-p}
	\end{pmatrix}. \]
	By dualizing back (and reordering the basis with respect to the weight filtration) one gets the $p$-adic periods
	\[ \mathcal{P}_p(K^\vee_a)=\begin{pmatrix}
		\frac{\log (a^{1-p})}{1-p}  \\
		1
	\end{pmatrix}. \]
	It is amusing to notice that the $p$-adic logarithm is not defined at $a$ in general (except if $a$ is $1$ modulo $p$) so that the $p$-adic number $\frac{\log (a^{1-p})}{1-p} $ is the best analog of the complex number $\log (a)$; i.e., passing from $\varphi(K_a)$ to $\mathcal{P}_p(K^\vee_a)$ corrects  the entries and gives the "right"  $p$-adic period. (The correction from  $ \log (a^{1-p}) $ to  $\frac{\log (a^{1-p})}{1-p} $ might also be obtained by an ad hoc change of basis, but recall that the above computations are made with respect to a canonical basis.)
	
	Finally, the groups $G_\dR( K^\vee_a)$ and $G_\crys( K^\vee_a)$ are respectively $\mathbb{G}_ a\rtimes \mathbb{G}_m $ and $\mathbb{G}_m$. In particular the $p$-adic Grothendieck period conjecture predicts that the transcendence degree of the $p$-adic periods of $K^\vee_a$ is $1$, i.e., that $\frac{\log (a^{1-p})}{1-p}$ is transcendental. This is the only non trivial case that we are aware of where the conjecture is actually known as it is follows from a theorem of Mahler \cite{Mahler} (see also Bertrand \cite[Theorem 1]{bertrand}).
	\end{example}

\begin{remark}\label{weight} (Pure vs mixed periods.)
In the conjectural abelian category of mixed motives, any motive is expected to have a weight filtration  which should split over finite fields.
In particular, there should be plenty of non trivial extensions
\[0 \rightarrow N  \rightarrow M \rightarrow \one \rightarrow 0\]
(with $N$ being a motive of negative weight) that split when reduced modulo $p$.
This leads to algebraic classes modulo $p$ that are not liftable, just as the example above, hence they provide a source of non trivial $p$-adic periods.
In the same spirit as the Kummer example, they are related to the coefficients of the matrix of the absolute Frobenius acting on crystalline cohomology.

Let us give another concrete example along these lines. Let $E$ be an elliptic curve over $\bQ$, $O$ be its origin and $P$ be a rational non torsion point.  The localization triangle in the triangulated category of motives 	gives a non trivial extension that  should correspond in the abelian category of motives to a short exact sequence
\[0 \rightarrow \one  \rightarrow   H^1(E,\{O,P\}) \rightarrow  H^1( E )  \rightarrow 0,\]
which we dualize 
\[0 \rightarrow H^1( E )^{\vee}   \rightarrow   H^1(E,\{O,P\})^{\vee} \rightarrow   \one \rightarrow 0,\]
to have the suited form as above  (with $N=H^1( E )^{\vee}$ and $M= H^1(E,\{O,P\})^{\vee}).$

The Frobenius matrix is of the shape
	\[ \varphi(H^1(E,\{O,P\})^{\vee})=\begin{pmatrix}
	a & b & c \\
    d & e & f \\
    0 & 0 & 1
\end{pmatrix} \]
and the $p$-adic period matrix is a vector of the shape
\[ \mathcal{P}_p(H^1(E,\{O,P\})^{\vee})=\begin{pmatrix}
	 \alpha \\
	  \beta \\
	1
\end{pmatrix}.\]
The relation between the two is given by
\[  \begin{pmatrix}
	a & b  \\
	d & e  
	 
\end{pmatrix} \cdot  \begin{pmatrix}
\alpha  \\
\beta 

\end{pmatrix} =
\begin{pmatrix}
	c  \\
	f
	
\end{pmatrix}.
\]

\

The general picture would be the following. Let $M$ be a motive and, for all integers $n$, let $M_n$ be its associated $n$-th graded piece for the weight filtration. 
As all extensions split modulo~$p$, one has
\[\spe(M)=\bigoplus_n \spe(M_n).\]
Moreover, for weight reasons, one has   $Z_p(\spe(M))= Z_p(\spe(M_0))$.

 Fix a basis for each $R_\dR(M_n)$ and  a splitting of the weight filtration of $R_\dR(M)$ and consider the induced basis of  $R_\dR(M)$.
For a given cycle  $\gamma \in Z_p(\spe(M))$, its coordinates with respect to this basis have the form
 	\[ \mathcal{P}_p(\gamma)=
	\begin{pmatrix}
		\vdots  \\
		v_n \\
		\vdots
	\end{pmatrix}\]
where $v_n $ is a vector living in weight $n$.

Notice that the coordinates of $v_0$ are  the pure $p$-adic periods of $\gamma$ seen as an element in $Z_p(\spe(M_0))$ and the total vector 
	$ \mathcal{P}_p(\gamma)$
	is the only Frobenius invariant vector whose zeroth coordinate is precisely $v_0$.
	
	In particular, mixed $p$-adic periods are generated by pure $p$-adic periods and the coefficients of the matrix of the Frobenius. (Recall also that there are more pure $p$-adic periods  than just the coefficients of the Frobenius matrix; see \cref{more than frob,Eg:two ell curves}.)
\end{remark}

\begin{remark}\label{remark ultima} (Frobenius matrix vs $p$-adic periods.)
As the above remarks and examples explained, the coefficients of the Frobenius matrix are intimately related to Andr\'e's $p$-adic periods. Some authors, especially in the context of mixed Tate motives, usually call $p$-adic periods those coefficients \cite{Furusho,Brownp}. Let us explain why Andr\'e's $p$-adic periods are actually better $p$-adic  analogues of classical complex periods.

First,  \cref{ex log} shows that the $p$-adic numbers appearing as Andr\'e's $p$-adic periods have a more striking analogy with the corresponding complex ones than the coefficients of the Frobenius have.  

Second, recall that for motives other than mixed Tate motives there are more Andr\'e $p$-adic periods  than just the coefficients of the Frobenius matrix (\cref{more than frob}).

Third, and more important, the motivic interpretation of Andr\'e's $p$-adic periods seems to behave as in the complex case, contrary to the coefficients of the Frobenius. Indeed we expect that the evalutation map
	$\eval_{B}\colon \cO(\cH_M)\to \bQb_p$ is injective; see Proposition \ref{P:equiv of GPC}. This is the analogous of the fact that the classical evalutation map for complex periods
	$\cO(\Isom(R_\dR, R_B)) \longrightarrow \bC$ is expected to be injective. On the other hand the evaluation map
	  $ \cO (\Aut (R_{\dR})) \longrightarrow \bQ_p$ considered by \cite{Furusho,Brownp}
	 is well known to be non injective at least because the   characteristic polynomial of the Frobenius has integral coefficients and this gives relations between the entries of the Frobenius that cannot be seen in $\Aut (R_{\dR})$. For example if one takes a non CM elliptic curve, the group $\Aut (R_{\dR})$ has dimension $4$ whereas the four coefficients of the Frobenius are related at least by trace and determinant; see \cref{Eg:H_M for E=ordinary}. In our setting we do not know if   $\eval_{B}\colon \cO(\cH_M)\to \bQb_p$ is injective   but we do know at least that the relations induced by the   characteristic polynomial have motivic origins as shown in the following proposition.
	 
\end{remark}
\begin{proposition}\label{pol char}
Let $M$ be a motive and $f\colon \spe(M) \rightarrow \spe(M) $ be an endomorphism of its specialization. We will see $f$ as an algebraic class in positive characteristic, more precisely as an element of $Z_p(M\otimes M^{\vee})$.

Suppose that the realization of $M$ is concentrated in one cohomological degree, hence that the characteristic polynomial of $f$ has algebraic coefficients \cite[Proposition 2.7]{GKL}.
Fix a basis $\{v_i\}$ of $R_\dR(M)$ and let $\{v_i^\vee\}$ be the dual basis. Consider the $p$-adic periods $c_{ij}=[M\otimes M^{\vee},f,v_i \otimes v_j^\vee]\in \bQb_p$ and the motivic $p$-adic periods $c^{\mathrm{mot}}_{ij} = [M\otimes M^{\vee},f,v_i \otimes v_j^\vee]^{\mathrm{mot}}\in \cAp(M\otimes M^{\vee}) $ introduced in \cref{period triples,motivic periods}. 

Then the characteristic polynomial 
\[\det(\{c^{\mathrm{mot}}_{ij}\}-T \cdot \id) \in \cAp(M\otimes M^{\vee}) [T]\supset \bQb[T]\]
actually lies in $\bQb[T]$ and is equal to  
$\det(\{c_{ij}\}-T \cdot \id),$
the characteristic polynomial of $f$.
\end{proposition}

\begin{proof}
By definition of $\cAp(M\otimes M^{\vee})$ we need to prove relations between the functions $c^{\mathrm{mot}}_{ij} $ on the variety $\cH_{M\otimes M^{\vee}} $ from
\cref{D:of H_M,T:cH_M is repr}.
In particular it is enough to show that these relations hold after evaluating at each point of the variety. 
If we evaluate at the Berthelot point $\fcB$ from \cref{BO point}, then the evaluation of $\det(\{c^{\mathrm{mot}}_{ij}\}-T \cdot \id) $ is exactly $\det(\{c_{ij}\} -T \cdot \id) $ because of \cref{surj second}. If we now evaluate at another point we can use that the action of the tannakian group $G_\dR(M)$ is transitive on $\cH_{M\otimes M^{\vee}} $ by \cref{T:cH_M is repr}. 
This action  sends the matrix $\{c_{ij}\} $ to a conjugate matrix, hence the characteristic polynomial is the same.
 \end{proof}

\begin{remark}\label{remark ultima 2} (Continuing \cref{remark ultima}.)
 In the same vein as the above proposition and remark notice that our constructions give a much better bound to the transcendence degree of the coefficients of the Frobenius matrix than the classical ones. If $M$ is a motive and $d_M$ is the transcendence degree of the algebra generated by those coefficients,  one gets $d_M \leq  \dim G_\dR(M)$ from the  constructions in \cite{Furusho,Brownp}  whereas one gets 
\begin{multline}
	d_M \leq \dim G_\dR(M\otimes M^\vee) -   \dim G_\crys(M\otimes M^\vee) \le \\\leq   \dim G_\dR(M) -   \dim G_\crys(M)\le \dim G_\dR(M)
\end{multline}
from \cref{T:trdeg <= dim group}.
All the above inequalities are, in general, strict.
The first one is an equality if and only if the $p$-adic periods of the graded pieces of the weight filtration of $M$ are trivial (see \cref{weight}).
The last one is an equality when $\dim G_\crys(M)=0$, namely  if and only if  $M$ becomes, modulo $p$, a sum  of copies of the unit motive.
The second inequality is more subtle: the motive $M\otimes M^\vee$ corresponds to a representation of $G_\dR(M)$ which might not be faithful. If its kernel is $K$, then we have 
\begin{align*}
G_\dR(M\otimes M^\vee) &= G_\dR(M)/K \quad \text{ and} \\
G_\crys(M\otimes M^\vee) &= G_\crys(M)/(K\cap G_\crys(M)).
\end{align*}
Hence we have equality if and only if the kernel $K$ is included in $G_\crys(M)$.

\end{remark}

	
	

\section{$\ell$-adic version}\label{section:ellGPC}
The aim of this section is to investigate $\ell$-adic versions of the $p$-adic questions we raised before. 
Let us immediately say that these $\ell$-adic counterparts turn out to be false in general. 
We will  still give also some positive results.

The arguments for the negative results are inspired by \cite{AndreBetti}.

\

We will consider the functors from  \cref{def:realization}  as well as the following ones.

\begin{definition}\label{D:lperiods}(cf. \cref{D:matrix of periods of M}.)
Define the following functors on $\MotgrQ$:

\begin{enumerate}
\item (Betti realization) Given a field embedding $\sigma: \obQ \hookrightarrow \bC$    define the Betti realization as 
$R_B=\rHB(\sigma^*(-),\bQ)$,
\item ($\ell$-adic realization) $R_\ell:=\rH\et(-,\bQl)$,
\item (Comparison) $\fc(M):  R_B(M) \otimes \bQ_\ell \simeq R_\ell(M) \simeq R_\ell(\spe(M)) $
where the first isomorphism is Artin's comparison and the second is smooth proper base change.
\end{enumerate}
Given a motive $M\in\Mot(\bQb)_\bQ$ fix a basis $\mathcal{B}$ of the $\bQ$-vector space $Z_p(\spe(M))$ and a basis $\mathcal{B}'$  of the $\bQ$-vector space $R_B(M)$.

We will write $\mathcal{P}_{\ell,\sigma}(M)$ for the  $\bQl$-matrix
 containing the coordinates of $\mathcal{B}$ with respect to $\mathcal{B}'$  and call its entries the $\ell$-adic periods of $M$.
\end{definition}

The $\ell$-adic analogs of \cref{Conj:liftp,Conj:p adic analog GPC} are:
\begin{question}\label{Conj:GPCWell}
For which $M,\sigma$ and $\ell$ is the inclusion 
\[Z_0(M)\subset Z_p(M) \cap R_B(M)
\]
an equality?
\end{question}

\begin{question}\label{Conj:GPCSell}
For which $M,\sigma$ and $\ell$ is the inequality
\[ \trdeg (\mathcal{P}_{\ell,\sigma}(M)) \leq \dim(G_B(M))-\dim(G_\ell(\spe(M)))
\]
an equality?
\end{question}

 \begin{remark}
 	One main difference between these questions and their complex or $p$-adic counterparts is the presence of the choice of $\sigma$. In particular the left hand side in the inclusion in \cref{Conj:GPCWell} and in the inequality in \cref{Conj:GPCSell}   depend a priori on $\sigma$ whereas the right hand sides do not depend. As we will see below, beside some exceptional cases, the left hand sides will actually depend on  $\sigma$ which will  force the above questions to have negative answers in some cases.
 	\end{remark}

\begin{proposition}\label{thmabel}
	Let $A$ be a simple CM abelian variety over a number field with good reduction modulo $p$ and assume moreover that the reduction $A_p$ is simple.
	Then an endomorphism of $A_p$  whose action on  $\rH^1_\ell(A_p)  \simeq \rH^1_\ell(A) \simeq  \rH^1_B(A) \otimes \bQ_\ell$   preserves the rational singular cohomology  $\rH^1_B(A)$  lifts  to an endomorphism  of $A$. 
\end{proposition}
\begin{remark}
Given a simple CM abelian variety over a number field, the set of prime numbers $p$ such that its reduction  modulo $p$ is simple has density one \cite[Theorem 3.1]{Murty}.
\end{remark}
\begin{proof}
Denote by $A_p$ the reduction modulo $p$ of $A$.
By hypothesis, there is a CM number field $F\subset \End(A)\otimes_\bZ\bQ$ acting transitively on $\rHB^1(A,\bQ)\setminus\{0\}$. 

Let $g$ be an endomorphism of $A_p$.
Let us suppose that $g$ acts rationally on $\rHB^1(A,\bQ)$.
Pick a non zero vector $v$  in  $\rHB^1(A,\bQ)$.
Because of the transitivity of the action of $F$, there is an $f\in F$ such that $g(v)=f(v)$.
This means that $g-f$ has non zero kernel on $\rHB^1(A,\bQ)$. 
Hence, as an endomorphism of $A_p$, $g-f$ has a non zero kernel. 
Such a kernel defines a subabelian variety of $A_p$ which was assumed to be simple.
Therefore we deduce that $g-f=0$, hence $g$ lifts to characteristic zero.
\end{proof}
\begin{corollary}\label{corabelell}
 Let $A$ be an abelian variety as in \cref{thmabel}  and $M$ be the motive $\mathfrak{h}^1(A) \otimes \mathfrak{h}^1(A)^{\vee}$ or the motive $\mathfrak{h}^2(A)(1)$. Then \cref{Conj:GPCWell} has a positive answer for $M$ (independently of $\sigma$ and $\ell$).
\end{corollary}
\begin{proposition}\label{thmsurf}
 \cref{Conj:GPCWell} has a positive answer for surfaces with maximal Picard rank (independently of $\sigma$ and $\ell$). 
\end{proposition}
\begin{proof}
Let $S$ be the surface in characteristic zero and let $S_p$ be its reduction to characteristic $p$.
The motive $\fh^2(S)(1)$ is endowed with a quadratic form that is realized in singular and $\ell$-adic cohomology as the cup product and in algebraic cycles as the intersection product, compatible under the class morphism and the comparison isomorphisms.

From \cite[14.2.3]{KMP} we have an orthogonal motivic decomposition
\[\mathfrak{h}^2(S)(1)=\one^{\oplus \rho}\oplus M\]
where $M$ is a motive whose realization is the  transcendental part of the Betti cohomology and $\rho=\rho(S)=\dim(\rH^2(S,\bQ)\cap \rH^{1,1}(S,\bC))$ is the Picard rank. 
The assumption $\rho(S)=h^{1,1}(S)$ implies that $\rHB(M)\otimes \bC= H^{2,0}(S,\bC)\oplus H^{0,2}(S,\bC)$ and moreover that the latter $\bC$-vector space is defined over $\bQ$ (see for example \cite[Proposition 1]{Beauville_surf_max_Pic}).

The statement we need to prove reduces to the fact that an algebraic class on $\spe(M)$ cannot be in $\rHB(M)$, i.e., that $Z_p(\spe(M))\cap\rHB(M)=\{0\}$, where the comparison isomorphism and the class morphism are used implicitly.

To show this fact we will use the above quadratic form restricted to  $M$.
The vector space $H^{2,0}(S,\bC)\oplus H^{0,2}(S,\bC)$ is defined over $\bR$ and the Hodge--Riemann relations say that the quadratic form is positive definite on it and therefore also on $\rHB(M)$.

On the other hand, the Hodge index theorem for $S_p$ tells us that this quadratic form has signature $(1,\rho_p-1)$ on the algebraic classes on $\fh^2(S_p)(1)$, i.e. on $Z_p(\spe(\one^{\oplus \rho}\oplus M))$. 
However, the positive part is already realized from algebraic cycles in $\one^{\oplus \rho}$ (for instance, by a hyperplane section) hence the quadratic form is negative definite on $Z_p(\spe(  M))$.

The $\bQ$-vector space $\rHB(M)\cap Z_p(\spe(M))$ has a quadratic form which is both positive and negative definite, this forces the intersection to be zero.
\end{proof}
\begin{proposition}
 \cref{Conj:GPCWell}  
  has a positive answer for polarized motives of rank two and of Hodge type $ (-i,i),(i,-i)$ with $i$ odd and it has a positive answer for the Fermat cubic fourfold  (independently of $\sigma$ and $\ell$).
\end{proposition}
\begin{proof}
The proof goes as in  \cref{thmsurf}.
The Hodge index theorem is replaced by the main result of \cite{Ancona}.
For the example of the Fermat cubic fourfold; see \cite[Proposition A.9]{Ancona}.
\end{proof}

We end the  section  with some negative answers to \cref{Conj:GPCWell} and   to \cref{Conj:GPCSell} in the case of elliptic curves.

\begin{proposition}\label{prop:SerrenonCM} Let $E$ be an elliptic curve without CM and $p$ be a prime of good reduction. 
Let $f$ be an endomorphism of $E$ modulo   $p$ which does not lift to characteristic zero. 
Then there exist an $\ell$ and a $\sigma$ such that the action of $f$ on $H^1_\ell(E)$ preserves the rational structure coming from singular cohomology. 

Hence, for such an $\ell$ and $\sigma$ and for the motive $M=\mathfrak{h}^1(E) \otimes \mathfrak{h}^1(E)^{\vee}$, \cref{Conj:GPCWell} has a negative answer.
\end{proposition}

\begin{proof} Let $K$ be a number field over which $E$ is defined and let $G$ be its absolute Galois group. 
For each $g$, write $g^*$ for the action of $g$ on the $\ell$-adic Tate module $T_\ell(E)$.
For each  field embedding $\sigma: \obQ \hookrightarrow {\bC}$  
write $\sigma^*$ for the  identification $T_\ell(E)=\Lambda \otimes_{\bZ} \bZ_\ell$ of the $\ell$-adic Tate module with singular homology $\Lambda:=\rH_1(E,\bZ)$. 
Notice that we have the relation
\begin{equation}\label{change sigma}
(\sigma \circ g)^*=\sigma^*\circ g^*.\end{equation}
By a  theorem of Serre \cite[(5) page 260]{SerreGL2}, for all but finitely many primes $\ell$, the group $G$ acts on $T_\ell(E)$ through $\GL(T_\ell(E))$. 
In particular, given such an $\ell$,  for any $\bZ_\ell$-basis $v,w$ of $T_\ell(E)$ there exists a $\sigma$ such that $v,w$ live in $\Lambda$. 
Hence, to conclude, it is enough to choose a suitable basis $v,w$ for which the matrix of $f$ is rational.  

Now $v,w$ form a $\bZ_\ell$ basis of $T_\ell(E)$ if and only if their reduction modulo $\ell$, denoted $\bar{v},\bar{w}$, form an $\mathbb{F}_\ell$-basis of $T_\ell(E)\otimes _{\bZ_\ell}\bF_\ell$.
For all but finitely many $\ell$ the determinant of $f$ is invertible modulo $\ell$. 
Moreover, the minimal polynomial of $f$ has rational coefficients and it is independent of $\ell$. By considering its discriminant we can choose moreover $\ell$ such that the minimal polynomial has distinct roots modulo $\ell$ (this is possible since $f$ does not lift hence it is not a multiple of the identity).
For such a choice of $\ell$, pick  any vector $v$  such that $\bar{v}$ is not an eigenvector of $f$ modulo $\ell$.
Then $v,f(v)$ form a $\bZ_\ell$ basis of $T_\ell(E)$.
With respect to this basis, $f$ acts as a companion matrix therefore it has rational coefficients.
\end{proof}

\begin{proposition} Let $E$ be an elliptic curve with CM and $p$ be a prime of supersingular reduction. 
Let $D$ be the endomorphism algebra of $E$ modulo   $p$. 
Then there exist an $\ell$ and a $\sigma$ such that the action of $D$ on $H^1_\ell(E)$  is through matrices whose entries are algebraic with respect to  the rational structure coming from singular cohomology. 

In particular, for such an $\ell$ and $\sigma$ and for the motive $M=\mathfrak{h}^1(E) \otimes \mathfrak{h}^1(E)^{\vee}$, \cref{Conj:GPCSell} has a negative answer.
\end{proposition}

\begin{proof}
Write $F$ for the CM quadratic number field. 
	Write the quaternion algebra $D=\<f,h\>$ with $f\in F$ satisfying $f^2=b\in\bZ_{<0}$ and $h^2=a\in \bQ^\times\setminus \bQ^{\times 2}$ with the relation $fh=-hf$.
	
	Consider the action of $D$ on $T_\ell(E)$. 
	The minimal polynomial of $h$ is $T^2-a$. We know that for half of the primes $\ell$ this polynomial splits over $\bF_\ell$ and hence over $\bZ_\ell$.
	Let $\ell$ be such a prime and let $v\in T_\ell(E)$ be an eigenvector for $h$. 
	Due to the relation $fh=-hf$, the vector $f(v)$ 
is also an eigenvector for $h$ but for the other eigenvalue.
	In particular $v,f(v)$ is a $\bQ_\ell$ basis of $ T_\ell(E)_{\bQ_\ell}$.
	In this basis, the action of $D$ on $T_\ell(E)_{\bQ_\ell}$ is through the following matrices:
	\begin{align}
	f\mapsto \begin{pmatrix}
	0 & b\\
	1 & 0
	\end{pmatrix} & &
	h\mapsto \begin{pmatrix}
	\alpha & 0\\
	0 & -\alpha
	\end{pmatrix}
	\end{align}
	where $\alpha\in\bZ_\ell$ is a square root of $a$, hence an algebraic number.
	
	We now claim that there exists $\sigma$ (and a suitable $\ell$) such that $v$ lives inside  the lattice of singular cohomology $\Lambda$ (and hence $f(v)$ does as well).
	This will conclude the proof as 	the algebra $D$ will act on $T_\ell(E)$ with respect to a basis of $\Lambda$ through algebraic numbers (more precisely, numbers in $\bQ[\sqrt a])$.

	To show the claim, we argue as in the proof of \cref{prop:SerrenonCM}. 
	Fix a $\sigma$ and consider the identification $T_\ell(E)=\Lambda \otimes_{\bZ} \bZ_\ell$. 
	We want to modify this $\sigma$ to obtain the one with the desired property using the group $G=\Gal(K)$ and the relation \eqref{change sigma}.
	
	Recall that  $T_\ell(E)$ is a free module of rank one over $F\otimes \bZ_\ell$ and that, moreover,
  for all but finitely many $\ell$, the group $G=\Gal(K)$ acts on $T_\ell(E)$ through $(F\otimes \bZ_\ell)^\times$ (this is due to Deuring; see \cite[\S 4.5]{SerreGL2} and the references therein).
  For such an $\ell$, there will be an element of $G$ that sends the vector $v$ to an element of $\Lambda$.
\end{proof}

\begin{remark}\label{finalrmk}
	In conclusion, despite some positive results, it does not seem reasonable to expect the existence of an $\ell$-adic  analogue of the Grothendieck period conjecture, or at least one independent of the embedding $\sigma: \obQ \hookrightarrow \bC$ as in Questions  \ref{Conj:GPCWell} and \ref{Conj:GPCSell}. Following the spirit of \cite[Proposition 2]{AndreBetti}, it seems reasonable to expect that for a generic choice of $\sigma$ these questions should have positive anwser and that they should be consequences of  the Mumford--Tate conjecture.
\end{remark}


\def\cprime{$'$}

\Addresses

\end{document}